\input ./style/arxiv-general.cfg
\documentclass[aap,MSNbibl,seceqn,dvips]{arximspdf}
\makeatletter
   \@ifpackageloaded{graphicx}{}{\usepackage{graphicx}}
\makeatother
\usepackage{mathbh}

\doi{10.1214/15-AAP1112}
\volume{26}
\issue{2}
\pubyear{2016}
\firstpage{1082}
\lastpage{1110}
\docsubty{FLA}

\makeatletter

\newcommand{\rrvert}{\vert}

\newcommand{\llvert}{\vert}
\newcommand{\eqref}[1]{(\ref{#1})}

\newtheorem{theorem}{Theorem}[section]
\newtheorem{lemma}{Lemma}[section]
\newtheorem{lemmaa}{Lemma}[section]
\newtheorem{corollary}{Corollary}[section]
\newtheorem{proposition}{Proposition}[section]
\newproclaim{definition}{Definition}[section]
\newproclaim{remark}{Remark}[section]
\newproclaim{example}[theorem]{Example}
\newproclaim{assumption}{Assumption}[section]

\newcommand{\eps}{\varepsilon}
\newcommand{\vp}{\varphi}

\newcommand{\M}{\mathbb{M}}
\newcommand{\R}{\mathbb{R}}
\newcommand{\F}{\mathbb{F}}
\renewcommand{\P}{\mathbb{P}}
\renewcommand{\L}{\mathbb{L}}
\newcommand{\mfU}{\mathfrak{U}}
\newcommand{\Ac}{\mathcal{A}}
\newcommand{\Uc}{\mathcal{U}}
\newcommand{\D}{\mathcal{D}}
\newcommand{\Fc}{\mathcal{F}}

\makeatother

\begin{document}
\begin{frontmatter}

\title{Stochastic Perron for stochastic target games\thanksref{T1}}
\runtitle{Stochastic Perron for stochastic target games}
\thankstext{T1}{Supported in part by  NSF  Grant DMS-09-55463.}

\begin{aug}
\author[A]{\fnms{Erhan}~\snm{Bayraktar}\corref{}\ead[label=e1]{erhan@umich.edu}}
\and
\author[A]{\fnms{Jiaqi}~\snm{Li}\ead[label=e2]{lijiaqi@umich.edu}}
\runauthor{E. Bayraktar and J. Li}
\affiliation{University of Michigan}
\address[A]{Department of Mathematics\\
University of Michigan\\
530 Church Street\\
Ann Arbor, Michigan 48109\\
USA\\
\printead{e1}\\
\phantom{E-mail:\ }\printead*{e2}}
\end{aug}

%
\received{\smonth{9} \syear{2014}}
%
\revised{\smonth{1} \syear{2015}}


\begin{abstract}
We extend the stochastic Perron method to analyze the framework of
stochastic target games, in which one player tries to find a strategy
such that the state process almost surely reaches a given target no
matter which action is chosen by the other player. Within this
framework, our method produces a viscosity sub-solution
(super-solution) of a Hamilton--Jacobi--Bellman (HJB) equation. We then
characterize the value function as a viscosity solution to the HJB
equation using a comparison result and a byproduct to obtain the
dynamic programming principle.
\end{abstract}

%
\end{frontmatter}

\section{Introduction}\label{sec1}
We will extend the stochastic Perron method to analyze a stochastic
(semi) game where a controller tries to find a strategy such that the
controlled state process almost surely \emph{reaches a given target} at
a given finite time, no matter which control is chosen by an adverse
player (nature). More precisely, the controller has access to a
filtration generated by a Brownian motion and can observe and react to
nature, who may choose a parametrization of the model to be totally
adverse to the controller, in a nonanticipative way. This stochastic
target game was introduced and analyzed in \cite{Bouchard_Nutz_TargetGames}.

In this paper, we will have a fresh look at the problem of Bouchard and
Nutz \cite{Bouchard_Nutz_TargetGames} with a different methodology,
namely the stochastic Perron method. Using this method we will be able
to drop the assumption on the concavity of the Hamiltonian assumed in
\cite{Bouchard_Nutz_TargetGames}. The stochastic Perron method was
introduced in \cite{Bayraktar_and_Sirbu_SP_LinearCase} for analyzing
linear problems, in \cite{Bayraktar_and_Sirbu_SP_DynkinGames} for
Dynkin games involving free-boundary games and in \cite
{Bayraktar_and_Sirbu_SP_HJBEqn} for stochastic control problems. This
method is a type of verification theorem, which identifies the value
function as the unique solution to a corresponding HJB equation without
going through the dynamic programming principle, but does not require
the smoothness of the value function. It is a stochastic version of the
Perron method \cite{MR1118699} in that it creates classes of sub- and
super-solutions that envelope the value function and are closed under
maximization and minimization, respectively. More recently, the
stochastic Perron method was adjusted to solve exit time problems in
\cite{MR3217159}, state constraint problems in~\cite{ECP3616}, singular
control problems in~\cite{2014arXiv1404.7406B}, stochastic games in
\cite{Sirbu_SP_elementary_strategy} and\vadjust{\goodbreak} control problems with model
uncertainty in \cite{Sirbu_SP_CompanionPaper} and~\cite
{2014arXiv1409.6233B}. In this paper, we show how the main ideas of
this method can be modified to analyze the stochastic target games of
Bouchard and Nutz \cite{Bouchard_Nutz_TargetGames}.

The main difficulty of this analysis is identifying the correct
collections of stochastic sub- and super-solutions. Once this is
established, the technical contribution is in showing that in fact the
supremum and the infimum of the respective families are viscosity
super- and sub-solutions, respectively. Then a comparison result
establishes the claim since the value function is already enveloped by
these two families. The identification of these classes and the
technical proofs turn out to be quite different from the works cited
above because of the difference between nature of the stochastic target
problems and the nature of the stochastic control problems. Unlike the
usual stochastic control problems, the goal of the target problems is
to beat a stochastic target almost surely by applying the admissible
controls. These problems, which are generalizations of the
super-hedging problems that appear in mathematical finance, were
introduced in the seminal papers \cite{MR1920265} and \cite{MR1924400};
see \cite{Touzi_book} for a more recent exposition. Stochastic target
games, on the other hand, were considered recently by Bouchard, Moreau
and Nutz \cite{MR3199977} when the target is of controlled loss type.
The more difficult case of an almost sure target was then analyzed in
\cite{Bouchard_Nutz_TargetGames}.

In this paper we achieve the following:
\begin{itemize}
\item We give a proof of the result that the value function of the
stochastic target game is the unique viscosity solution of the
associated HJB equation without first going through the geometric
dynamic programming principle. What we have is a new method for
analyzing stochastic target problems.
%
\item We give a more elementary proof of the result in \cite
{Bouchard_Nutz_TargetGames}. This way we are able to avoid using
Krylov's method of shaken coefficients, which requires the concavity of
the Hamiltonian.
\end{itemize}

The rest of the paper is organized as follows: In Section~\ref
{sec:prob}, we present the
setup of the stochastic target game, introduce the related HJB equation
and the definitions of the sets of stochastic super- and sub-solutions
(our conceptual contribution). The technical contribution of the paper
is given in Section~\ref{sec:main}, where we characterize the infimum
(supremum) of the stochastic super-solutions (sub-solutions) as the
viscosity sub-solution (super-solution) of the HJB equation. A
viscosity comparison argument concludes that the value function is the
unique bounded continuous viscosity solution of the HJB equation.
Finally, we obtain the dynamic programming principle as a byproduct.
Some technical results are deferred to the \hyperref[app]{Appendix}.

\section{Statement of the problem}\label{sec:prob}
\subsection{The value function}
Let us denote
\[
\D:=[0,T]\times\R^{d}, \qquad\mathcal{D}_{<T}:=[0,T)\times
\R^{d}, \qquad\mathcal{D}_{T}:=\{T\} \times\R^{d}.
\]
Let $\Omega$ be the space of continuous functions $\omega\dvtx
[0,T]\to
\R
^{d}$, and let $\P$ be the Wiener measure on $\Omega$.
We will denote by $W$ the canonical process on $\Omega$, that is,
$W_t(\omega)=\omega_t$, and by $\F=(\Fc_s)_{0\leq s\leq T}$ the
augmented filtration generated by~$W$. For $0\leq t\leq T$ let $\F
^t=(\Fc^t_s)_{0\leq s\leq T}$ be the augmented filtration generated by
$(W_s-W_t)_{s\geq t}$. By convention, $\Fc^t_s$ is trivial for $s\leq t$.

We denote by $\Uc^{t}$ (resp., $\Ac^{t}$) the collection of all $\F
^{t}$-predictable processes in $L^p(\P\otimes\,dt)$ with values in a
given Borel subset $U$ (resp., bounded set $A$) of $\R^{d}$, where
$p\geq2$ is fixed.

Given $(t,x,y)\in\D\times\R$ and $(u,\alpha)\in\Uc^t\times\Ac^t$,
consider the stochastic differential equations (SDEs)
%
\begin{equation}\qquad
\label{eq: dynamics} \cases{ %
 dX(s)=\mu_{X}
\bigl(s,X(s),\alpha_{s}\bigr) \,ds +\sigma_{X}\bigl(s,X(s),
\alpha_{s}\bigr)\,dW_{s},
\vspace*{2pt}\cr
dY(s)=\mu_{Y}\bigl(s,X(s),Y(s),u_{s},
\alpha_{s}\bigr) \,ds +\sigma_{Y}\bigl(s,X(s),Y(s),u_{s},
\alpha_{s}\bigr)\,dW_{s},}
\end{equation}
with initial data $(X(t),Y(t))=(x,y)$.
%
\begin{assumption} \label{Assump: drift_and_volatility}
The coefficients $\mu_{X}, \mu_{Y}, \sigma_{X}$ and $\sigma_{Y}$ are
continuous in all variables and take values in $\R^{d}$, $\R$, $\R^{d}$
and $\M^{d}:=\R^{d\times d}$, respectively. There exists $K>0$
such that
\begin{eqnarray*}
\label{eq: cond mu sigma} %
\bigl |\mu_{X}(t,x,\cdot)-
\mu_{X}\bigl(t',x',\cdot\bigr)\bigr|+\bigl|
\sigma_{X}(t,x,\cdot)-\sigma_{X}\bigl(t',x',
\cdot\bigr)\bigr|&\le& K\bigl(\bigl|t-t'\bigr|+\bigl|x-x'\bigr|\bigr),
\\
\bigl|\mu_{X}(\cdot,x,\cdot)\bigr|+\bigl|\sigma_{X}(\cdot,x,\cdot)\bigr|&\le&
K,
\\
\bigl|\mu_{Y}(\cdot,y, \cdot)-\mu_{Y}\bigl(\cdot,y', \cdot\bigr)\bigr| + \bigl|\sigma_{Y}(\cdot,y, \cdot)-
\sigma_{Y}\bigl(\cdot,y', \cdot\bigr)\bigr| &\le& K
\bigl|y-y'\bigr|,
\\
\bigl|\mu_{Y}(\cdot,y, u,\cdot)\bigr|+\bigl|\sigma_{Y}(\cdot,y,u,
\cdot)\bigr|&\le& K\bigl(1+|u|+|y|\bigr),
\end{eqnarray*}
for all $(x,y),(x',y')\in\R^d\times\R$ and $u\in U$.
\end{assumption}
This assumption ensures that the stochastic differential equations
given in \eqref{eq: dynamics} are well posed. Denote the solutions to
\eqref{eq: dynamics} by $(X^{\alpha}_{t,x}, Y^{u,\alpha}_{t,x,y})$.
Let $t\le T$. We say that a map $\mathfrak{u}\dvtx\Ac^{t}\to\Uc^{t}$,
$\alpha
\mapsto\mathfrak{u}[\alpha]$ is a $t$-admissible strategy if it is
nonanticipating in the sense that
\[
\label{eq: cond non anticipativity} \bigl\{\omega\in\Omega\dvtx
\alpha(\omega){|_{[t,s]}}=
\alpha'(\omega){|_{[t,s]}}\bigr\}\subset\bigl\{\omega\in
\Omega\dvtx\mathfrak{u}[\alpha](\omega) {|_{[t,s]}}= \mathfrak
{u}\bigl[
\alpha'\bigr](\omega) {|_{[t,s]}}\bigr\} \mbox{-a.s.}
\]
for all $s\in[t,T]$ and $\alpha,\alpha'\in\Ac^t$, where $|_{[t,s]}$
indicates the restriction to the interval $[t,s]$. We denote by $\mfU
(t)$ the collection of all $t$-admissible strategies; moreover, we
write $Y^{\mathfrak{u},\alpha}_{t,x,y}$ for $Y^{\mathfrak{u}[\alpha
],\alpha}_{t,x,y}$.
Then we can introduce the value function of the stochastic target game,
%
\begin{equation}\qquad
\label{eq: def v} v(t,x):=\inf\bigl\{y\in\R\dvtx\exists\mathfrak
{u}\in\mfU(t) \mbox{ s.t. } Y^{\mathfrak{u},\alpha}_{t,x,y}(T)\ge g\bigl(X^{\alpha}_{t,x}(T)
\bigr) \mbox{-a.s.}\ \forall\alpha\in\Ac^{t} \bigr\},
\end{equation}
where $g\dvtx\R^d\to\R$ is a bounded and measurable function. We
also need
to define strategies starting at a family of stopping times. Let
$\mathbb{S}^t$ be the set of $\mathcal{F}^t$-stopping times valued in $[t,T]$.

\begin{definition}[(Nonanticipating family of stopping times)] \label
{def: nonanticipating family_of_stopping times}\break 
Let $\{\tau^{\alpha}\}_{\alpha\in\Ac^t} \subset\mathbb{S}^t$ be a
family of stopping times. This family is $t$-\break nonanticipating if
\begin{eqnarray*}
\label{eq: nonanticipating family_of_stopping times} %
&&\bigl\{\omega\in\Omega\dvtx\alpha(\omega){|_{[t,s]}}=
\alpha'(\omega){|_{[t,s]}}\bigr\}
\\
&&\qquad\subset\bigl\{\omega\in\Omega\dvtx t\leq\tau^{\alpha
}(\omega)=\tau^{\alpha^{\prime}}(\omega)\leq s \bigr\} \cup\bigl\{\omega\in\Omega\dvtx s<\tau
^{\alpha}(\omega), s<\tau^{\alpha^{\prime}}(\omega)\bigr\} \mbox{-a.s.}%
\end{eqnarray*}
Denote the set of $t$-nonanticipating families of stopping times by
$\mathfrak{S}^t$.
\end{definition}
We will use $\{\tau^{\alpha}\}$ for short to represent $\{
\tau
^{\alpha}\}_{\alpha\in\Ac^t}$, which will always denote a
$t$-nonanticipating family of stopping times.
%
\begin{definition}[(Strategies starting at a nonanticipating family of
stopping times)] \label{def:strategies at a family of stopping times}
Fix $t$, and let
$\{\tau^{\alpha}\} \in\mathfrak{S}^t$. We say that a map $\mathfrak
{u}\dvtx\Ac
^{t} \to\Uc^{t}$, $\alpha\mapsto\mathfrak{u}[\alpha]$ is a $(t,\{
\tau^{\alpha
}\})$-admissible strategy if it is nonanticipating in the sense that
\begin{eqnarray*}
\label{eqn:strategies at a family of stopping times} %
&&\bigl\{\omega\in\Omega\dvtx\alpha(\omega){|_{[t,s]}}= \alpha'(\omega){|_{[t,s]}}\bigr
\}\\
&&\qquad\subset\bigl\{\omega\in\Omega\dvtx s<\tau^{\alpha}(\omega), s<
\tau^{\alpha^{\prime}}(\omega)\bigr\} \\
&&{}\qquad\cup
\bigl\{\omega\in\Omega\dvtx t\leq\tau^{\alpha}(\omega)=
\tau^{\alpha^{\prime}}(\omega) \leq s,\\
&&\hspace*{33pt} \mathfrak{u}[\alpha](\omega)
{|_{[\tau^{\alpha}(\omega),s]}}= \mathfrak{u}\bigl[\alpha'\bigr
](\omega)
{|_{[\tau
^{\alpha^{\prime}}(\omega),s]}}\bigr\} \mbox{-a.s.}%
\end{eqnarray*}
for all $s\in[t,T]$ and $\alpha,\alpha'\in\Ac^t$, denoted by
$\mathfrak{u}
\in\mfU(t,\{\tau^{\alpha}\})$.
\end{definition}
It is clear that from Definition~\ref{def:strategies at a
family of stopping times} that if we set $\tau^{\alpha}=t$ for all
$\alpha$, then $ \mfU(t,\{\tau^{\alpha}\})$ is then the same as
$\mfU
(t)$. Hence the above definitions are consistent.

\begin{definition}[(Concatenation)]
Let $\alpha_1, \alpha_2 \in\Ac^t$, $\tau\in\mathbb{S}^t$ be a
stopping time. The concatenation of $\alpha_1, \alpha_2$ is defined
as follows:
\[
\alpha_1\otimes_{\tau} \alpha_2:=
\alpha_1 \mathbh{1}_{[t,\tau)}+ \alpha_2
\mathbh{1}_{[\tau,T]}.
\]
The concatenation of elements in $\Uc^t$ is defined in a similar fashion.
\end{definition}
%
\begin{lemma} \label{lemma: justification of concatenation of strategies}
Fix $t$, and let $\{\tau^{\alpha}\}\in\mathfrak{S}^t$. For
$\mathfrak{u}\in
\mfU(t)$ and $\tilde{\mathfrak{u}} \in\mfU(t,\{\tau^{\alpha}\}
)$, define $\mathfrak{u}
_{*}[\alpha]:= \mathfrak{u}[\alpha] \otimes_{\tau^{\alpha}}
\tilde{\mathfrak{u}}[\alpha
]$.
Then $\mathfrak{u}_{*}\in\mfU(t)$. For the rest of the paper, we
will use
$\mathfrak{u}\otimes_{\tau^{\alpha}} \tilde{\mathfrak{u}}[\alpha
]$ to represent $\mathfrak{u}
[\alpha] \otimes_{\tau^{\alpha}} \tilde{\mathfrak{u}}[\alpha]$.
\end{lemma}
\begin{pf}
It is obvious that $u_{*}$ maps $\Ac^t$ to $\Uc^t$. Let us check the
nonanticipativity of the map. For any fixed $s\in[t,T]$ and $\alpha
,\alpha'\in\Ac^t$, $\omega'\in\{\omega\in\Omega\dvtx \alpha
(\omega
){|_{[t,s]}}= \alpha'(\omega){|_{[t,s]}}\} $, by Definition~\ref{def:
nonanticipating family_of_stopping times},
%
\begin{equation}
\label{eq: concatenation_1} \omega' \in\bigl\{ t\leq\tau^{\alpha}=
\tau^{\alpha^{\prime}}\leq s \bigr\} \cup\bigl\{ s<\tau^{\alpha}, s<
\tau^{\alpha^{\prime}}\bigr\} \mbox{-a.s.}
\end{equation}
(i) If $\omega' \in\{ t\leq\tau^{\alpha}=\tau^{\alpha^{\prime}}\leq s \} $,
by the definition of $\mathfrak{u}_{*}$,
\begin{eqnarray*}
\mathfrak{u}_{*}[\alpha]\bigl(\omega'
\bigr){|_{[t,s]}}&=&\mathfrak{u}[\alpha]\bigl(\omega'\bigr)
\mathbh{1}_{[t,\tau^{\alpha}(\omega'))}{|_{[t,s]}}+ \tilde
{\mathfrak{u}}[\alpha]
\bigl(\omega'\bigr) \mathbh{1}_{[\tau^{\alpha}(\omega'),T]} {|_{[t,s]}},
\\
\mathfrak{u}_{*}\bigl[\alpha'\bigr]\bigl(\omega'\bigr){|_{[t,s]}}&=&\mathfrak{u}\bigl[\alpha
'\bigr]\bigl(\omega'\bigr) \mathbh{1}_{[t,\tau^{\alpha^{\prime}}(\omega
'))}{|_{[t,s]}}+
\tilde{\mathfrak{u}}\bigl[\alpha'\bigr]\bigl(\omega'
\bigr) \mathbh{1}_{[\tau^{\alpha^{\prime}}(\omega'),T]} {|_{[t,s]}}.
\end{eqnarray*}
Since $\tau^{\alpha}(\omega')=\tau^{\alpha^{\prime}}(\omega')$,
$\mathfrak{u}\in\mfU
(t)$ and by Definition~\ref{def:strategies at a family of stopping
times}, we know
\[
\omega' \in\bigl\{\omega\in\Omega\dvtx\mathfrak{u}[\alpha
](\omega)
{|_{[t,s]}}= \mathfrak{u} \bigl[\alpha'\bigr](\omega)
{|_{[t,s]}}\bigr\} \mbox{-a.s.}
\]
(ii) If $\omega' \in\{ s<\tau^{\alpha}, s<\tau^{\alpha^{\prime}}\}$, using
the definition of $\mathfrak{u}_{*}$,
\begin{eqnarray*}
\mathfrak{u}_{*}[\alpha]\bigl(\omega'
\bigr){|_{[t,s]}}&=&\mathfrak{u}[\alpha]\bigl(\omega'
\bigr){|_{[t,s]}},
\\
\mathfrak{u}_{*}\bigl[\alpha'\bigr]\bigl(\omega'\bigr){|_{[t,s]}}&=&\mathfrak{u}\bigl[\alpha
'\bigr]\bigl(\omega'\bigr){|_{[t,s]}}.
\end{eqnarray*}
Since $\omega'\in\{\omega\in\Omega\dvtx \alpha(\omega){|_{[t,s]}}=
\alpha
'(\omega){|_{[t,s]}}\} $ and $\mathfrak{u}\in\mfU(t)$, then
$\omega'\in\{\omega\in\Omega\dvtx \break  \mathfrak{u}_{*}[\alpha
](\omega)
{|_{[t,s]}}=
\mathfrak{u}_{*}[\alpha'](\omega) {|_{[t,s]}}\} \mbox{-a.s.}$
\end{pf}

\subsection{The HJB equation}

Before giving the HJB equation, we will introduce some notation and an
assumption, which was also assumed in \cite{Bouchard_Nutz_TargetGames}.
Given $(t,x,y,z,a)\in\D\times\R\times\R^{d}\times A$, define the set
\[
N(t,x,y,z,a):=\bigl\{u\in U\dvtx\sigma_{Y}(t,x,y,u,a)=z\bigr\}.
\]
%
\begin{assumption}\label{ass: def hat u + regu} $u\mapsto\sigma
_{Y}(t,x,y,u,a)$ is invertible. More precisely, there exists a
measurable map $\hat u\dvtx\D\times\R\times\R^{d} \times A\to U$
such that
$ N=\{\hat u\}$. Moreover, the map $\hat u(\cdot, a)$ is continuous for
each $a\in A$.
\end{assumption}
Let us define for $(t,x,y,p,M)\in\D\times\R\times\R
^{d}\times\M^{d}$,
\begin{eqnarray*}
\label{eq: Hamitonianian} &&H(t,x,y,p,M):=\sup_{a \in A} \biggl\lbrace-
\mu^{\hat
u}_{Y}\bigl(t,x,y,\sigma_{X}(t,x,a)p,a
\bigr)+\mu_{X}(t,x,a)^{\top} p \\
&&\hspace*{202pt}{}+ \frac{1}2
\operatorname{Tr}\bigl[\sigma_{X}\sigma_{X}^{\top
}(t,x,a)M
\bigr] \biggr\rbrace,
\end{eqnarray*}
where
\[
\mu^{\hat u}_{Y}(t,x,y,z,a):=\mu_{Y}\bigl(t,x,y,
\hat u(t,x,y,z,a),a\bigr), \qquad z\in\R^{d}.
\]
Consider the equation
%
\begin{eqnarray}
\label{HJB equation} %
\phi_t+H\bigl(t,x,
\phi,D\phi,D^2\phi\bigr)&=&0\qquad \mbox{on } \D_{<T},
\nonumber
\\[-8pt]
\\[-8pt]
\nonumber
\phi&=&g \qquad\mbox{on } \D_{T}.
\end{eqnarray}

%
\subsection{Stochastic solutions}
We will introduce weak solution concepts to the HJB equation that are
stable under minimization and maximization, respectively, and envelope
the value function $v$ of the stochastic target game.
%
\begin{definition} [(Stochastic super-solutions)] \label{def:
Stochasticsuper-solution}
A function $w\dvtx[0,T]\times\mathbb{R}^d \rightarrow\mathbb{R}$ is
called a stochastic super-solution of (\ref{HJB equation}) if:
\begin{longlist}[(1)]
\item[(1)] it is bounded, continuous and $w(T, \cdot)\geq g(\cdot)$;
\item[(2)] for fixed $(t,x,y)\in\D\times\mathbb{R}$ and $\{\tau^{\alpha
}\}
\in\mathfrak{S}^t$, for any $\mathfrak{u}\in\mfU(t)$, there exists a
strategy $\tilde{\mathfrak{u}}\in\mfU(t,\{\tau^{\alpha}\})$ such
that for any
$\alpha\in\mathcal{A}^t$ and each stopping time $\rho\in\mathbb
{S}^t$, $\tau^{\alpha}\leq\rho\leq T$ with the simplifying notation
$X:= X_{t,x}^{\alpha}, Y:=Y_{t,x,y}^{\mathfrak{u}\otimes_{\tau
^{\alpha}}\tilde
{\mathfrak{u}}[\alpha],\alpha}$, we have
\[
Y(\rho)\geq w\bigl(\rho, X(\rho)\bigr) \qquad\mathbb{P}\mbox{-a.s.}
\mbox{ on } \bigl
\{Y\bigl(\tau^{\alpha}\bigr)> w\bigl(\tau^{\alpha}, X\bigl(\tau^{\alpha}\bigr)\bigr)\bigr\}.
\]
\end{longlist}
\end{definition}
The set of stochastic super-solutions is denoted by $\Uc^+$. Assume it
is nonempty and $v^{+}:=\inf_{w \in\Uc^{+}} w$. For any stochastic
super-solution $w$, choose $\tau^{\alpha} =t$ for all $\alpha$ and
$\rho
=T$. Then there exists $\tilde{\mathfrak{u}}\in\mfU(t)$ such that,
for any $
\alpha\in\mathcal{A}^t$,
\[
Y_{t,x,y}^{\tilde{\mathfrak{u}},\alpha}(T)\geq w \bigl(T,
X_{t,x}^{\alpha
}(T)
\bigr)\geq g \bigl(X_{t,x}^\alpha(T) \bigr)\qquad \mathbb{P}\mbox{-a.s.} \mbox{ on } \bigl\{y>w(t,x)\bigr\}.
\]
Hence, $y>w(t,x)$ implies $y\geq v(t,x)$ from \eqref{eq: def v}. This
gives $w\geq v$ and $v^{+}\geq v$. Similarly, we could define the
stochastic sub-solutions.
%
\begin{definition} [(Stochastic sub-solutions)] \label{def:
Stochasticsub-solution}
A function $w\dvtx[0,T]\times\mathbb{R}^d \rightarrow\mathbb{R}$ is
called a stochastic sub-solution of (\ref{HJB equation}) if:
\begin{longlist}[(1)]
\item[(1)] it is bounded, continuous and $w(T, \cdot)\leq g(\cdot)$;
\item[(2)] for fixed $(t,x,y)\in\D\times\mathbb{R}$ and $\{\tau^{\alpha
}\}
\in\mathfrak{S}^t$, for any $\mathfrak{u}\in\mfU(t)$, $\alpha\in
\mathcal
{A}^t$, there
exists $\widetilde{\alpha}\in\Ac^t$ (may depend on $\mathfrak{u}$,
$\alpha$
and $\tau^{\alpha}$) such that for each stopping time $\rho\in
\mathbb
{S}^t$, $\tau^{\alpha}\leq\rho\leq T$ with the simplifying notation
$X:= X_{t,x}^{\alpha}, Y:=Y_{t,x,y}^{\mathfrak{u},\alpha\otimes
_{\tau^{\alpha
}}\widetilde{\alpha}}$, we have
\[
\mathbb{P} \bigl(Y(\rho)< w \bigl(\rho, X(\rho) \bigr)| B \bigr)>0,
\]
for any $B \subset\{Y(\tau^{\alpha})
< w(\tau^{\alpha}, X(\tau^{\alpha
}))\}$, $B \in
\Fc^t_{\tau^{\alpha}}$ and $\mathbb{P}(B)>0$.
\end{longlist}
\end{definition}
The set of stochastic sub-solutions is
denoted by $\Uc^-$. Assume it is nonempty, and let $v^{-}:=\sup
_{w \in\Uc^{-}} w$. For any stochastic sub-solution $w$, choose $
\tau^{\alpha} =t$ for all $\alpha$ and $\rho=T$. Hence for any $
\mathfrak{u}\in\mfU(t)$, there exists $\widetilde{\alpha}\in
\mathcal{A}^t$,
such that %
\[
\mathbb{P} \bigl(Y_{t,x,y}^{\mathfrak{u},\widetilde{\alpha
}}(T)< w \bigl(T,
X_{t,x}^{\widetilde{\alpha}}(T) \bigr)\leq g\bigl
(X_{t,x}^{\widetilde
{\alpha
}}(T)
\bigr) | y<w(t,x) \bigr)>0.
\]
Hence, $y<w(t,x)$ implies $y\leq v(t,x)$ from \eqref{eq: def v}.
This gives $w\leq v$ and $v^{-}\leq v$. As a result we have
%
\begin{equation}\label{eq:intfvmavp} v^- \triangleq\sup_{w\in
\mathcal{U}^-}
w\leq v \leq\inf_{w\in
\mathcal{U}^+}w\triangleq v^+.
\end{equation}
We will show in Section~\ref{sec:main} that under some suitable assumptions, $v^{+}$
and $v^{-}$ are viscosity sub- and super-solutions of (\ref{HJB
equation}), respectively.

\subsection{Additional technical assumptions}

We will need to make some more technical assumptions as in \cite
{Bouchard_Nutz_TargetGames}.

\begin{assumption}\label{ass: regu muY_hatu} The map
$(t,x,y,z)\in\D\times\R\times\R^{d}\mapsto\mu_{Y}^{\hat
u}(t,x,y,z,a)$ is Lipschitz continuous, uniformly in $a\in A$, and
$(y,z)\in\R\times\R^{d}\mapsto\mu_{Y}^{\hat u}(t,x,y,z,a)$ has linear
growth, uniformly in $(t,x,a)\in\D\times A$.
\end{assumption}

For the derivation of the super-solution property of $v^-$, we will
impose a condition on the growth of $\mu_Y$ relative to $\sigma_Y$.
%
\begin{assumption}\label{assump: relative_growth_condition_mu_to_sigma}
\[
\sup_{u\in U}\frac{|\mu_{Y}(\cdot,u,\cdot)|}{1+\|\sigma_{Y}(\cdot
,u,\cdot)\|} \qquad\mbox{is locally bounded},
\]
where $\|\cdot\|$ is the Euclidean norm.
\end{assumption}
%

In \eqref{eq:intfvmavp} we implicitly assume that the sets $\Uc^{+}$
and $\Uc^{-}$ are nonempty. The assumptions we made already imply that
$\Uc^{+}$ is not empty, but the same may not be true when $\Uc^{-}$ is
not empty.
%
\begin{assumption}\label{ass:colunepty}
The collection $\Uc^{-}$ is not empty.
\end{assumption}

\subsection{When $\mathcal{U}^+$ and $\mathcal{U}^-$ are not empty}

As the next result shows, the assumptions above already guarantee that
$\Uc^{+}$ is not empty.
%
\begin{proposition}\label{prop:Upnemty}
Under Assumptions~\ref{Assump: drift_and_volatility}, \ref{ass: def hat
u + regu} and \ref{ass: regu muY_hatu} the collection $\mathcal{U}^+$
is not empty.
\end{proposition}
\begin{pf}
See the \hyperref[app]{Appendix}.
\end{pf}
In the above proposition the assumptions made can be replaced by the
following natural assumption (although this is not the route we will take):

\begin{assumption} \label{assum: stochastic semisolution is nonempty}
There exists $\mathbf{u}\in U$ such that $ \mu_Y(t,x,y,\mathbf
{u},a)=0$, $\sigma_Y(t,x,y,\mathbf{u},a)=0$ for all $(t,x,y,a)\in
\mathcal{D}_{<T}
\times\mathbb{R} \times A $. (In these equations the right-hand sides
are denoted by just 0 for simplicity, but they in fact are collections
of 0's matching the dimension on the left-hand side.)
\end{assumption}

In the context of super-hedging in mathematical finance, in which $Y$
represents the wealth of an investor and $X$ the stock price, and
$g(X_T)$ a financial contract, the last assumption is equivalent to
allowing the investor not to trade in the risky assets.

\begin{proposition}\label{prop:ancfup}
Under Assumptions~\ref{Assump: drift_and_volatility} and \ref{assum:
stochastic semisolution is nonempty} the collection $\mathcal{U}^+$ is
not empty.
\end{proposition}

\begin{pf}
Choose
the strategy $\tilde{u}[\alpha]=\mathbf{u}$. For any given $\{\tau
^{\alpha}\} \in\mathfrak{S}^t$, we have $\tilde{\mathfrak{u}} \in
\Uc(t,\{\tau
^{\alpha}\})$, and from Assumption~\ref{assum: stochastic semisolution
is nonempty}, it holds for any $\mathfrak{u}\in\mfU(t)$ that
\[
Y_{t,x,y}^{\mathfrak{u}\otimes_{\tau^{\alpha}}\tilde{\mathfrak
{u}}[\alpha],\alpha}(\rho)=Y_{t,x,y}^{\mathfrak{u}\otimes_{\tau
^{\alpha}}\tilde{\mathfrak
{u}}[\alpha],\alpha
}\bigl(\tau^{\alpha}\bigr)\qquad \forall\alpha\in\Ac^t \mbox{ and } \rho\in
\mathbb{S}^t \mbox{ such that } \tau^{\alpha}\leq\rho\leq T.
\]
From the boundedness of $g$, there exists a $C$, such that $g(x)<C$.
Now take $w(t,x)\equiv C$, which clearly satisfies the first condition
in Definition~\ref{def: Stochasticsuper-solution}. On the other hand,
on the set $\{Y(\tau^{\alpha})>w(\tau^{\alpha},X(\tau^{\alpha}))\}
$, we
clearly have that $\{Y(\rho)>w(\rho,X(\rho))\}$ for any $\rho$ such
that $\tau^{\alpha}\leq\rho\leq T$, which gives the second condition
in Definition~\ref{def: Stochasticsuper-solution}.
\end{pf}

\begin{proposition}\label{prop:uminempy}
If in addition to Assumptions~\ref{Assump: drift_and_volatility} there
exists $a\in A$ such that $\mu_Y(t,x,y,u,a)=0$, $\sigma_Y(t,x,y,u,a)=0$
for all $(t,x,y,u)\in\mathcal{D}_{<T}\times\mathbb{R} \times U $,
then $\mathcal
{U}^-$ is not empty.
\end{proposition}

\begin{pf}
The proof is similar to that of Proposition~\ref{prop:ancfup}.
\end{pf}

The additional assumption in the latter proposition is not very
reasonable. Below we introduce an alternative assumption.
%
\begin{assumption} \label{assum: stochastic semisolution is nonemptyp}
$\frac{|\mu_{Y}|}{\|\sigma_{Y}\|}$ is bounded on $N=\{
(t,x,y,u,a)\dvtx\sigma_{Y} (t,x,y, u,a)\neq0\}$.
\end{assumption}

\begin{proposition}\label{Prop: U+ notempty}
Under Assumptions~\ref{Assump: drift_and_volatility}, \ref{ass: def hat
u + regu}, \ref{assum: stochastic semisolution is nonempty} and \ref
{assum: stochastic semisolution is nonemptyp}, the collection $\Uc^{-}$
is not empty.
\end{proposition}

\begin{pf}
See the \hyperref[app]{Appendix}.
\end{pf}

\section{The main result and its proof}\label{sec:main}
To prove the main theorem, we need some preparatory lemmas.
%
\begin{lemma} The set of stochastic super/sub solutions is
upwards/\break downwards directed; that is:
\begin{longlist}[(1)]
\item[(1)] if $w_1, w_2\in\mathcal{U}^+$, then $w_1\wedge w_2\in\mathcal{U}^+$;
\item[(2)] if $w_1,w_2\in\mathcal{U}^-$, then $w_1\vee w_2\in\mathcal{U}^-$.
\end{longlist}
\end{lemma}
\begin{pf}
This lemma is in the spirit of Lemma~3.7 in \cite
{Sirbu_SP_elementary_strategy}. Here we only sketch the proof for
$(1)$. For $w_1, w_2\in\mathcal{U}^+$, let $w=w_1\wedge w_2$. Clearly
$w$ is bounded, continuous and $w(T,x)\geq g(x)$. For fixed $(t,x,y)\in
\mathcal{D}_{<T}\times\mathbb{R}$ and $\{\tau^{\alpha}\}\in
\mathfrak{S}^t$, let
$\mathfrak{u}_1$ and $\mathfrak{u}_2$ be the strategies starting at
$\{\tau^{\alpha}\}$
for $w_1$ and $w_2$, respectively. Let
\[
\mathfrak{u}[\alpha]=\mathfrak{u}_1[\alpha]\mathbh{1}_{\{w_1(\tau
^{\alpha},X(\tau
^{\alpha}))<w_2(\tau^{\alpha},X(\tau^{\alpha}))\}}+
\mathfrak{u}_2[\alpha] \mathbh{1}_{\{w_1(\tau^{\alpha},X(\tau
^{\alpha}))\geq w_2(\tau
^{\alpha
},X(\tau^{\alpha}))\}}.
\]
It is easy to show that $\mathfrak{u}$ works for $w$ in the definition of
stochastic super-solutions.
\end{pf}
%
\begin{lemma}There exists a nonincreasing sequence $\mathcal{U}^+\ni
w_n\searrow v^+$ and a nondecreasing sequence $\mathcal{U}^-\ni
v_n\nearrow v^-$.
\end{lemma}
\begin{pf}
The proof of the lemma follows directly from Proposition~4.1 in \cite
{Bayraktar_and_Sirbu_SP_LinearCase}.
\end{pf}
Let us also state the following well-known result without proof.
%
\begin{lemma} \label{lemma: continuity of taking sup}
Given $f$: $X \times Y \subset\mathbb{R}^p \times\mathbb{R}^q
\rightarrow\mathbb{R}$, define $F(x):=\sup_{y\in Y}f(x, \break y)$. If
$x\rightarrow f(x,y)$ is continuous, uniformly in $y$ and $F(x)<\infty$
for all $x\in X$, then $x\rightarrow F(x)$ is continuous.
\end{lemma}

\begin{theorem}[(Stochastic Perron for stochastic target games)]\label
{thm: main theorem} Let Assumptions~\ref{Assump: drift_and_volatility}
and \ref{ass: def hat u + regu} hold.
\begin{longlist}[(1)]
\item[(1)] If in addition g is upper semi-continuous (USC) and
Assumption~\ref{ass: regu muY_hatu} holds, the function $v^+$ is a
bounded USC
viscosity sub-solution of (\ref{HJB equation}).
\item[(2)] On the other hand if g is lower semi-continuous (LSC) and
Assumptions~\ref{assump: relative_growth_condition_mu_to_sigma} and
\ref
{ass:colunepty}
hold, the function $v^-$ is a bounded LSC viscosity super-solution of
(\ref{HJB equation}).
\end{longlist}
\end{theorem}

\begin{pf}
\textit{Step \textup{1} \textup{(}$v^+$ is the viscosity sub-solution\textup{)}}. First due to
Proposition~\ref{prop:Upnemty} $v^+$ is well defined. We will first
show the interior viscosity sub-solution property and then demonstrate
the boundary condition.

\textit{Step \textup{1.1.} The interior sub-solution property}:
Let $(t_0,x_0)$ be in the parabolic interior $[0,T)\times\mathbb{R}^d$
such that a smooth function $\varphi$ strictly touches $v^+$ from above
at $(t_0,x_0)$. Assume, by contradiction, that
\[
\varphi_t +H\bigl(t,x,\varphi, D\varphi, D^2\varphi
\bigr)<0 \qquad\mbox{at }(t_0, x_0).
\]
From the uniform continuity of $\mu_{X}$ and $\sigma_{X}$ in
Assumption~\ref{Assump: drift_and_volatility}, the uniform continuity
of $\mu
_{Y}^{\hat u}$ in Assumption~\ref{ass: regu muY_hatu} and the
smoothness of $\varphi$, the map $(t,x,y,a)\rightarrow-\mu^{\hat
u}_{Y}(t,x,y,\sigma_{X}(t,x,a)D\varphi(t,x),a)+\mu_{X}(t,x,a)^{\top}
D\vp
+ \frac{1}2 \operatorname{Tr}[\sigma_{X}\times\break \sigma_{X}^{\top}(t,x,a) D^2\vp
(t,x)]$ is
uniformly continuous in $(t,x,y)$. Hence the map
$(t,\break x,y)\rightarrow H(t,x,y,D\varphi(t,x),D^2\varphi(t,x))$ is
continuous due to Lemma~\ref{lemma: continuity of taking sup}. This
implies that there exists a $\eps>0$ and $\delta>0$ such that
%
\begin{eqnarray}
\label{eq: sub_solution_locally_contradicted} \varphi_t +H\bigl
(t,x,y,D\varphi, D^2
\varphi\bigr)<0
\nonumber
\\[-8pt]
\\[-8pt]
\eqntext{\forall(t,x)\in\overline{B(t_0, x_0,
\eps)} \mbox{ and } \bigl|y-\vp(t,x)\bigr|\leq\delta,}
\end{eqnarray}
where $B(t_0,x_0,\eps)=\{(t,x)\in\D\dvtx\max\{|t-t_0|,|x-x_0|\}
<\eps\}
$. Now, on the compact torus $\mathbb{T}= \overline{B(t_0, x_0, \eps)}-
B(t_0, x_0, \eps/2)$, we have that $\varphi>v^+$, and the min of
$\varphi-v^+$ is attained since $v^{+}$ is USC. Therefore, $\varphi
>v^+ +\eta$ on $\mathbb{T}$ for some $\eta>0$. Since $w_n\searrow
v^+$, a Dini-type argument shows that for large enough $n$, we have
$\varphi>w_n+\eta/2$ on $\mathbb{T}$ and $\varphi>w_n-\delta$ on
$\overline{B(t_0, x_0, \eps/2)}$. For simplicity, fix such an $n$, and
denote $w=w_n$. Now define, for small $\kappa<\frac{\eta}{2} \wedge
\delta$,
\[
w^{\kappa}\triangleq\cases{ %
(\varphi-\kappa)
\wedge w,&\quad $\mbox{on } \overline{B(t_0, x_0, \eps)},$
\vspace*{2pt}\cr
w,&\quad $\mbox{outside } \overline{B(t_0, x_0, \eps)}.$
}
\]
Since $\vp>w+\eta/2>w+\kappa$ on $\mathbb{T}$, then $w=w^{\kappa}$ on
$\partial B(t_0,x_0;\eps/2)$, which implies $w^{\kappa}$ is continuous.
Since $w^{\kappa}(t_0,x_0)<v^+(t_0,x_0)$, we would obtain a
contradiction if we can show $w^{\kappa}\in\mathcal{U}^+$.

Fix $t$, $\{\tau^{\alpha}\} \in\mathfrak{S}^t $ and $\mathfrak
{u}\in\mfU(t)$.
We need to construct a strategy $\tilde{\mathfrak{u}}\in\mfU(t,\{
\tau^{\alpha
}\})$ in the definition of stochastic super-solutions for $w^{\kappa}$.
This can be done as follows:
since $w$ is a stochastic super-solution, there exists an ``optimal''
strategy $\tilde{\mathfrak{u}}_1$ in Definition~\ref{def:
Stochasticsuper-solution} for $w$ starting at $\{\tau^{\alpha}\}$. We
will construct $\tilde{\mathfrak{u}}$ in two steps:

\begin{longlist}[(ii)]
\item[(i)] $w^{\kappa}(\tau^{\alpha},X_{t,x}^{\alpha}(\tau
^{\alpha
}))= w(\tau^{\alpha},X_{t,x}^{\alpha}(\tau^{\alpha}))$: set
$\tilde{\mathfrak{u}
}=\tilde{\mathfrak{u}}_1$;

\item[(ii)] $w^{\kappa}(\tau^{\alpha},X_{t,x}^{\alpha}(\tau
^{\alpha
}))< w(\tau^{\alpha},X_{t,x}^{\alpha}(\tau^{\alpha}))$: In this
case we
necessarily start inside the ball.
Let $\overline{Y}$ be the unique strong solution (which is thanks in
particular to Assumption~\ref{ass: regu muY_hatu}) of the equation
\begin{eqnarray*}
\overline{Y}(l)&=&Y^{\mathfrak{u},\alpha}_{t,x,y}\bigl(\tau^\alpha\bigr)\\
&&{}+\int_{\tau^\alpha
}^{\tau^\alpha\vee l}
\mu_{Y}^{\hat{u}} \bigl(s,X_{t,x}^{\alpha
}(s),
\overline{Y}(s), \sigma_{X}\bigl(s,X_{t,x}^{\alpha}(s),
\alpha_{s}\bigr)D\vp\bigl(s,X_{t,x}^{\alpha}(s)
\bigr),\alpha_{s} \bigr) \,ds
\\
&&{}+ \int_{\tau^\alpha}^{\tau^\alpha\vee l} \sigma_{X}
\bigl(s,X_{t,x}^{\alpha
}(s),\alpha_{s}\bigr)D\vp
\bigl(s,X_{t,x}^{\alpha}(s)\bigr)\,dW_{s},\qquad l \geq\tau
^{\alpha},
\end{eqnarray*}
for any $\mathfrak{u}\in\mfU(t)$ and $\alpha\in\Ac^t$, and set
$\overline
{Y}(s)=Y_{t,x,y}^{\mathfrak{u}, \alpha}(s) $ for $s <\tau^{\alpha
}$. Define
\[
\tilde{\mathfrak{u}}_0:=\tilde{\mathfrak{u}}_0[
\alpha](s)=\hat{u}\bigl(s,X_{t,x}^{\alpha
}(s),\overline{Y}(s),
\sigma_{X}\bigl(s,X_{t,x}^{\alpha}(s),\alpha
_s\bigr)D\varphi\bigl(s,X_{t,x}^{\alpha}(s)\bigr),
\alpha_s\bigr).
\]
Let $\theta_{1}^{\alpha}$ is the first exit time of
$(s,X_{t,x}^{\alpha
}(s))$ after $\tau^{\alpha}$ from $ B(t_0, x_0; \eps/2)$ and $\theta
_2^{\alpha}$ be the first time after $\tau^{\alpha}$ when
$|\overline
{Y}(s)-\varphi(s,X_{t,x}^{\alpha}(s))| \geq\delta$. More precisely,
\[
\theta_{1}^{\alpha}:=\inf\bigl\{s \in\bigl[
\tau^{\alpha},T\bigr]\dvtx\bigl(s, X_{t,x}^{\alpha}(s)
\bigr) \notin B(t_0, x_0, \eps/2) \bigr\}
\]
and
\[
\theta_2^{\alpha}:=\inf\bigl\{s \in\bigl[
\tau^{\alpha},T\bigr]\dvtx\bigl\llvert\overline{Y}(s)-\varphi
\bigl(s,X_{t,x}^{\alpha}(s)\bigr)\bigr\rrvert\geq\delta\bigr\}.
\]
Let $\theta^{\alpha}=\theta_1^{\alpha} \wedge\theta_2^{\alpha}$. We
know that $\{\theta^{\alpha}\} \in\mathfrak{S}^t$ from Example~1 in
\cite{Bayraktar_Huang}. We will set $\tilde{\mathfrak{u}}$ to be
$\tilde{\mathfrak{u}}_0$
until $\theta^{\alpha}$. Starting at $\theta^{\alpha}$, we will then
follow the strategy $\mathfrak{u}^{\theta} \in\mfU(t,\{\theta
^{\alpha}\})$
which is ``optimal'' for $w$.
\end{longlist}

In summary, (i) and (ii) together give us the following strategy:
\[
\tilde{\mathfrak{u}}[\alpha]= \bigl(\mathbh{1}_{A} \tilde{
\mathfrak{u}}_1[\alpha]+\mathbh{1}_{A^c}\bigl(\tilde{
\mathfrak{u}}_0[\alpha] \mathbh{1}_{[t,\theta
^{\alpha})} +
\mathfrak{u}^{\theta}[\alpha]\mathbh{1}_{[\theta
^{\alpha
},T]}\bigr) \bigr)
\mathbh{1}_{[\tau^{\alpha}, T]},
\]
where
\[
A=\bigl\{ w^{\kappa}\bigl(\tau^{\alpha},X_{t,x}^{\alpha}
\bigl(\tau^{\alpha}\bigr)\bigr)= w\bigl(\tau^{\alpha
},X_{t,x}^{\alpha}
\bigl(\tau^{\alpha}\bigr)\bigr)\bigr\}.
\]
We note that $\tilde{\mathfrak{u}}_0 \in\mfU(t)$ by the pathwise
uniqueness of
$X$'s, $Y$'s and $\overline{Y}$'s equations. Then applying Lemma~\ref
{lemma: justification of concatenation of strategies}, $\tilde
{\mathfrak{u}
}_0[\alpha]\mathbh{1}_{[t,\theta^{\alpha})} + \mathfrak{u}^{\theta
}[\alpha
]\mathbh{1}_{[\theta^{\alpha},T]} \in\mfU(t) $.
Since $\tilde{\mathfrak{u}}_1 \in\mfU(t,\{\tau^{\alpha}\})$, by
Definition~\ref
{def:strategies at a family of stopping times}, it follows that $\tilde
{\mathfrak{u}}\in\mfU(t,\{\tau^{\alpha}\})$ by the pathwise
uniqueness of
$X$'s equation.
Now, let us show the above construction actually works. We need to show
that for any $\rho\in\mathbb{S}^t$ such that $\tau^{\alpha}\leq
\rho
\leq T$,
\[
Y(\rho)\geq w\bigl(\rho, X(\rho)\bigr)\qquad \mathbb{P}\mbox{-a.s.}
\mbox{ on } \bigl
\{Y\bigl(\tau^{\alpha}\bigr)> w\bigl(\tau^{\alpha}, X\bigl(\tau^{\alpha}\bigr)\bigr)\bigr\},
\]
where
\[
X:= X_{t,x}^{\alpha} \quad\mbox{and}\quad Y:=Y_{t,x,y}^{\mathfrak{u}\otimes
_{\tau
^{\alpha}}\tilde{\mathfrak{u}}[\alpha],\alpha}.
\]
Note that $\overline{Y}(s)= Y_{t,x,y}^{\mathfrak{u}\otimes_{\tau
^{\alpha}}\tilde
{\mathfrak{u}}_0[\alpha],\alpha}(s)
$ for $s\geq\tau^{\alpha}$ and
%
\begin{equation}
\label{eq: process_Y_on_two_disjoint_sets} Y=\mathbh{1}_A
Y_{t,x,y}^{\mathfrak{u}\otimes_{\tau^{\alpha
}}\tilde{\mathfrak{u}
}_1[\alpha],\alpha} +
\mathbh{1}_{A^c} Y_{t,x,y}^{\mathfrak
{u}\otimes_{\tau
^{\alpha}}\tilde{\mathfrak{u}}_0[\alpha],\alpha}\qquad \mbox{for } \tau
^{\alpha}\leq s \leq\theta^{\alpha}.
\end{equation}
We will carry out the proof in two steps:

\textit{\textup{(i)} On the set $A\cap\{Y(\tau^{\alpha})>w^{\kappa}(\tau
^{\alpha},X(\tau^{\alpha}))\}$}, we have
\[
Y\bigl(\tau^{\alpha}\bigr)>w\bigl(\tau^{\alpha},X\bigl(\tau^{\alpha}\bigr)\bigr).
\]
From \eqref{eq: process_Y_on_two_disjoint_sets} and the ``optimality''
of $\tilde{\mathfrak{u}}_1$ (for $w$), we know
\[
Y(\rho)=Y_{t,x,y}^{\mathfrak{u}\otimes_{\tau^{\alpha}}\tilde
{\mathfrak{u}}_1[\alpha
],\alpha}(\rho)\geq w\bigl(\rho,X(\rho)\bigr)\geq
w^{\kappa}\bigl(\rho,X(\rho)\bigr)\qquad\hspace*{-1pt} \mathbb{P}\mbox{-a.s} \mbox
{ on the above
set}.
\]

\textit{\textup{(ii)} On the set $A^c \cap\{Y(\tau^{\alpha})>w^{\kappa
}(\tau
^{\alpha},X(\tau^{\alpha}))\}$}, by the definition of $\tilde
{\mathfrak{u}}_0$
and \eqref{eq: process_Y_on_two_disjoint_sets}, using It\^{o}'s formula,
\[
Y\bigl(\cdot\wedge\theta^{\alpha}\bigr)- \vp\bigl(\cdot\wedge
\theta^{\alpha
},X\bigl(\cdot\wedge\theta^{\alpha}\bigr)\bigr)= Y
\bigl(\tau^{\alpha}\bigr)- \vp\bigl(\tau^{\alpha
},X\bigl(\tau
^{\alpha}\bigr)\bigr)+\int_{\tau^{\alpha}}^{\cdot\wedge\theta
^{\alpha}} \gamma
(s) \,ds,
\]
where
\begin{eqnarray*}
\gamma&:=&\mu_{Y}^{\hat{u}}\bigl(\cdot,X,Y,\sigma_{X}(\cdot,X,\alpha)D\varphi(\cdot,X),\alpha\bigr)-\mu_{X}(\cdot,X,
\alpha)^{\top} D\varphi(\cdot,X) \\
&&{}- \tfrac{1}2 \operatorname{Tr}\bigl[
\sigma_{X}\sigma_{X}^{\top}(\cdot,X,\alpha
)D^2\varphi(\cdot,X)\bigr]-\varphi_t(\cdot,X),
\end{eqnarray*}
since the definition of $\hat{u}$ allows us to cancel the Brownian
motion terms on the right-hand side. On $[\tau^{\alpha},\theta
^{\alpha
}], (t,X)\in\overline{B(t_0, x_0, \eps)} \mbox{ and } |Y(t)-\vp
(t,X(t))|\le\delta$, therefore from (\ref{eq:
sub_solution_locally_contradicted}) we have that $\gamma>0$. This
implies that $Y(\cdot\wedge\theta^{\alpha})- \vp(\cdot\wedge
\theta
^{\alpha},X(\cdot\wedge\theta^{\alpha}))$ is nondecreasing on
$[\tau
^{\alpha},T]$ and
%
\begin{equation}
\label{eq: increasing_process_tau_theta} Y\bigl(\theta^{\alpha
}\bigr)- \vp\bigl(\theta^{\alpha},X
\bigl(\theta^{\alpha
}\bigr)\bigr)+\kappa>Y\bigl(\tau^{\alpha}\bigr)-
\vp\bigl(\tau^{\alpha},X\bigl(\tau^{\alpha}\bigr)\bigr)+\kappa>0.
\end{equation}
As a result, on the one hand, we have
%
\begin{eqnarray}
\label{eq:theta1_less_theta2} 0<\bigl(Y\bigl(\theta_1^{\alpha
}\bigr)- \vp
\bigl(\theta_1^{\alpha},X\bigl(\theta_1^{\alpha
}
\bigr)\bigr)+\kappa\bigr)\leq\bigl(Y\bigl(\theta_1^{\alpha}
\bigr)- w\bigl(\theta_1^{\alpha
},X\bigl(\theta
_1^{\alpha}\bigr)\bigr)\bigr)
\nonumber
\\[-8pt]
\\[-8pt]
 \eqntext{\mbox{on } \bigl\{
\theta_1^{\alpha}<\theta_2^{\alpha}\bigr\}.}
\end{eqnarray}
On the other hand,
\[
Y\bigl(\theta_2^{\alpha}\bigr)-\vp\bigl(\theta_2^{\alpha},X
\bigl(\theta_2^{\alpha}\bigr)\bigr) = \delta\qquad\mbox{on } \bigl\{
\theta_1^{\alpha}\geq\theta_2^{\alpha}\bigr
\}.
\]
Observe that the right-hand side of the above expression cannot be
$-\delta$ due to \eqref{eq: increasing_process_tau_theta}. Therefore,
%
\begin{eqnarray}
\label{eq:theta1_greater_theta2} \bigl(Y\bigl(\theta_2^{\alpha
}\bigr)- w\bigl(\theta_2^{\alpha},X\bigl(\theta_2^{\alpha
}
\bigr)\bigr)\bigr)=\bigl(\delta+\varphi\bigl(\theta_2^{\alpha},X
\bigl(\theta_2^{\alpha
}\bigr)\bigr)-w\bigl(\theta
_2^{\alpha},X\bigl(\theta_2^{\alpha}\bigr)
\bigr)\bigr)>0
\nonumber
\\[-8pt]
\\[-8pt]
\eqntext{\mbox{on } \bigl\{\theta_1^{\alpha}\geq
\theta_2^{\alpha}\bigr\},}
\end{eqnarray}
since $\varphi>w-\delta$ on $\overline{B(t_0, x_0, \eps/2)}$. Combining
\eqref{eq:theta1_less_theta2} and \eqref{eq:theta1_greater_theta2} we obtain
%
\begin{equation}
\label{eq:aits} Y\bigl(\theta^{\alpha}\bigr)- w\bigl(\theta
^{\alpha},X
\bigl(\theta^{\alpha}\bigr)\bigr)>0 \qquad\mbox{on } A^c \cap\bigl\{Y
\bigl(\tau^{\alpha}\bigr)>w^{\kappa}\bigl(\tau^{\alpha
},X^{\alpha}
\bigr)\bigr\}.
\end{equation}
It follows from this conclusion and the ``optimality'' of $\mathfrak
{u}^{\theta
}$ starting at $\{\theta^{\alpha}\}$ that
\[
\bigl(Y\bigl(\rho\vee\theta^{\alpha}\bigr)- w^{\kappa}\bigl
(\rho
\vee\theta^{\alpha
},X\bigl(\rho\vee\theta^{\alpha}\bigr)\bigr) \bigr)
\geq\bigl(Y\bigl(\rho\vee\theta^{\alpha}\bigr)- w\bigl(\rho
\vee
\theta^{\alpha},X\bigl(\rho\vee\theta^{\alpha}\bigr)\bigr)
\bigr)
\geq0,
\]
on $A^c \cap\{Y(\tau^{\alpha})>w^{\kappa}(\tau^{\alpha},X^{\alpha
})\}$.

Also, since $ Y(\cdot\wedge\theta^{\alpha})- \vp(\cdot\wedge
\theta
^{\alpha},X(\cdot\wedge\theta^{\alpha}))$ is nondecreasing on
$[\tau
^{\alpha},T]$ it follows that $ (Y(\rho\wedge\theta^{\alpha
})-\vp
(\rho\wedge\theta^{\alpha},X(\rho\wedge\theta^{\alpha}))+\kappa
)>0$, which further gives
%
\begin{eqnarray}
\label{eq:eans} \bigl(Y\bigl(\rho\wedge\theta^{\alpha}\bigr
)-w^{\kappa}
\bigl(\rho\wedge\theta^{\alpha
},X\bigl(\rho\wedge\theta^{\alpha}
\bigr)\bigr) \bigr)>0
\nonumber
\\[-8pt]
\\[-8pt]
\eqntext{\mbox{on } A^c \cap\bigl\{ Y\bigl(\tau^{\alpha}\bigr)>w^{\kappa}\bigl(\tau^{\alpha},X^{\alpha}
\bigr)\bigr\}.}
\end{eqnarray}
From \eqref{eq:aits} and \eqref{eq:eans} we have
\[
Y(\rho)-w^{\kappa}\bigl(\rho,X(\rho)\bigr)\geq0 \qquad\mbox{on } A^c
\cap\bigl\{ Y\bigl(\tau^{\alpha}\bigr)>w^{\kappa}\bigl(\tau^{\alpha},X^{\alpha}\bigr)\bigr\}.
\]

\textit{Step \textup{1.2}. The boundary condition}:

\emph{Step} A: In this step we will assume that\vspace*{2pt} $\mu
_{Y}^{\hat
u}$ is nondecreasing in its $y$-variable. Assume on the contrary that
for some $x_{0} \in\mathbb{R}^d$, we have
%
\begin{equation}
\label{eq: step1.2_contra at the boundary} v^{+}(T,x_{0})>g(x_{0}).
\end{equation}
Since $g$ is USC, then from (\ref{eq: step1.2_contra at the boundary})
there exists $\eps>0$ such that
%
\begin{equation}
\label{eq: step1.2_contra at the boundary conclusion}
v^{+}(T,x_{0})>g(x)+\eps\qquad\mbox{for }
|x-x_0|\leq\eps.
\end{equation}
Choose $\varepsilon$ such that $\varepsilon<1$. Since $v^{+}$ is USC,
then $v^{+}$ is bounded above on the compact (rectangular) torus
$\mathbb{T}=\overline{B(T,x_{0};\eps)}-B(T,x_{0};\eps/2)$, where
$B(T,x_{0};\eps)=\{(t,x)\in\D\dvtx\max{\{|T-t|,|x-x_{0}|\}} < \eps
\}$.
Choose $\beta>0$ small enough, such that
\[
v^{+}(T,x_{0})+\frac{\eps^2}{4\beta}>\eps+\sup
_{\mathbb{T}}v^{+}(t,x).
\]
By a Dini-type argument there exists a $w\in\Uc^{+}$ such that
%
\begin{equation}
\label{Eq: 1.2_Dini_Argument} v^{+}(T,x_{0})+\frac{\eps^2}{4\beta
}>\eps+\sup
_{\mathbb{T}}w(t,x).
\end{equation}
For $C>0$ let us denote
\[
\varphi^{\beta,C}(t,x)=v^{+}(T,x_0)+
\frac{|x-x_0|^2}{\beta}+C (T-t).
\]
Hence, $D\vp^{\beta,C}(t,x)=\frac{2(x-x_0)}{\beta}$ and $D^2\vp
^{\beta
,C}(t,x)=\frac{2}{\beta}I_{d\times d}$. From Assumption~\ref{eq: cond
mu sigma},
%
\begin{eqnarray}
\label{eq: boundedness_1_v+_boundary} \bigl\llvert\mu
_{X}(t,x,a)^{\top} D
\vp^{\beta,C}(t,x)\bigr\rrvert\leq2K\frac
{|x-x_0|}{\beta}\leq
\frac{2K}{\beta}
\nonumber
\\[-8pt]
\\[-8pt]
\eqntext{\mbox{for } (t,x)\in\overline{B(T,x_{0};\eps)}
\mbox{ and } a\in A,}
\end{eqnarray}
where we use $\varepsilon<1$. Similarly,
%
\begin{eqnarray}
\label{eq: boundedness_2_v+_boundary} \biggl\llvert\frac
{1}{2}\operatorname{Tr} \bigl[
\sigma_{X}\sigma_{X}^{\top
}(t,x,a)D^{2}
\vp^{\beta,C}(t,x) \bigr]\biggr\rrvert\leq\frac
{1}{2}K^2
\frac
{2d}{\beta}=\frac{K^2d}{\beta}
\nonumber
\\[-8pt]
\\[-8pt]
\eqntext{\mbox{for } (t,x)\in\overline{B(T,x_{0};
\eps)} \mbox{ and } a\in A,}
\end{eqnarray}
where $d$ is the dimension of the space where the variable $x$ lives.
From the linear growth condition of $\mu_Y^{\hat u}$ in
Assumption~\ref
{ass: regu muY_hatu}, there exists a $L>0$, such that
%
\begin{eqnarray}
\label{eq: boundedness_3_v+_boundary} %
&&-\mu^{\hat u}_{Y}
\bigl(t,x,\vp^{\beta,0}-\eps,\sigma_{X}(t,x,a)D\vp
^{\beta
,0},a\bigr)\nonumber\\
&&\qquad\leq L \bigl(1+\bigl\llvert\vp^{\beta,0}(t,x)-\eps
\bigr\rrvert+\bigl\llvert\sigma_{X}(t,x,a)D\vp^{\beta,0}(t,x)
\bigr\rrvert\bigr)
\nonumber
\\[-8pt]
\\[-8pt]
\nonumber
&&\qquad\leq L \bigl(1+ v^+(T,x_0)+1/\beta+1+2K/\beta\bigr)\\
\eqntext{\mbox{for }
(t,x)\in\overline{B(T,x_{0};\eps)} \mbox{ and } a\in A.}
\end{eqnarray}
Noting that $D\vp^{\beta,C}(t,x)=D\vp^{\beta,0}(t,x)$, from the
monotonicity assumption of $\mu_Y^{\hat u}$, we have
\begin{eqnarray*}
&&-\mu^{\hat u}_{Y}\bigl(t,x,\vp^{\beta,C}-\eps,
\sigma_{X}(t,x,a)D\vp^{\beta
,C},a\bigr)\\
&&\qquad\leq-
\mu^{\hat u}_{Y}\bigl(t,x,\vp^{\beta,0}-\eps,\sigma
_{X}(t,x,a)D\vp^{\beta,0},a\bigr).
\end{eqnarray*}
The above equation, together with \eqref{eq:
boundedness_1_v+_boundary}, \eqref{eq: boundedness_2_v+_boundary} and
\eqref{eq: boundedness_3_v+_boundary}, implies that $H(\cdot,  \vp
^{\beta
,C}-\eps, D\vp^{\beta,C},D^{2}\vp^{\beta,C})(t,x)$ is bounded from
above on $\overline{B(T,x_{0};\eps)}$, and the bound is independent of
$C$. As a result for a large enough $C$ we have that
%
\begin{eqnarray}
\label{eq: contra 2_proof1.2_stepA} \vp_{t}^{\beta,C}+H\bigl(\cdot,y,D
\vp^{\beta,C},D^{2}\vp^{\beta,C}\bigr) (t,x) < 0
\nonumber
\\[-8pt]
\\[-8pt]
\eqntext{\forall
(t,x)\in B(T,x_{0};\eps) \mbox{ and } y\ge\vp^{\beta,C}(t,x)-
\eps,}
\end{eqnarray}
where we used the monotonicity assumption of $\mu_{Y}^{\hat u}$.
Making sure that $C \geq\eps/2 \beta$, we obtain from \eqref{Eq:
1.2_Dini_Argument} that
\[
\vp^{\beta,C}\geq\eps+w \qquad\mbox{on } \mathbb{T}.
\]
Also,
%
\begin{equation}
\label{eq: comparison_varphi_and_g} \vp^{\beta, C}(T,x)\geq
v^{+}(T,x_0)>g(x)+
\eps\qquad\mbox{for } |x-x_{0}|\leq\eps.
\end{equation}
Now we can choose $\kappa<\eps$ and define
%
\begin{equation}
\label{eq: w_kappa_1.2} w^{\beta,C,\kappa}\triangleq\cases{ %
 \bigl(\vp^{\beta,C} -\kappa\bigr)\wedge w,&\quad $\mbox{on } \overline{ B(T,
x_{0}, \eps)},$
\vspace*{2pt}\cr
w,&\quad $\mbox{outside } \overline{ B(T, x_0, \eps)}.$}
\end{equation}

From \eqref{eq: comparison_varphi_and_g} and \eqref{eq: w_kappa_1.2} it
is easy to see that $w^{\beta,C,\kappa}(T,x)\geq g(x)$. By applying
similar arguments as in step 1.1, we can show that $w^{\beta,C,\kappa}$
is a stochastic super-solution with $ w^{\beta,C,\kappa
}(T,x_{0})<v^{+}(T,x_{0}) $. This contradicts the definition of
$v^{+}$.

\emph{Step} B: We now turn to showing the same result for more
general $\mu_{Y}^{\hat u}$ and follow a proof similar to that in \cite
{Bouchard_Nutz_TargetGames}. Fix $c>0$, and define $\widetilde
Y_{t,x,y}^{\mathfrak{u},\alpha}$ as the strong solution of
\[
d\widetilde{Y}(s)=\tilde\mu_{Y}\bigl(s,X_{t,x}^{\alpha}(s),
\widetilde{Y}(s),\mathfrak{u}[\alpha]_{s}, \alpha_{s}\bigr)
\,ds +\tilde\sigma_{Y}\bigl(s,X_{t,x}^{\alpha}(s),
\widetilde{Y}(s),\mathfrak{u}[\alpha]_{s},\alpha_{s}
\bigr)\,dW_{s}
\]
with initial data $\widetilde{Y}(t)=y$, where
\begin{eqnarray*}
\tilde\mu_{Y}(t,x,y,u,a)&:=& c y+e^{ct}
\mu_{Y}\bigl(t,x,e^{-c t}y,u,a\bigr),
\\
\tilde\sigma_{Y}(t,x,y,u,a)&:=&e^{c t}
\sigma_{Y}\bigl(t,x,e^{-c t} y,u,a\bigr).
\end{eqnarray*}
Hence, $\widetilde{Y}_{t,x,y}^{\mathfrak{u},\alpha
}(s)e^{-cs}=Y_{t,x,ye^{-ct}}^{\mathfrak{u},\alpha}(s)$ for any $s\in
[t,T]$ by
the strong uniqueness. Set $\widetilde{g}(x):=e^{c T} g(x)$, and define
\[
\tilde v(t,x):= \inf\bigl\{y\in\R\dvtx\exists\mathfrak{u}\in
\mfU^{t}
\mbox{ s.t. } \widetilde{Y}^{\mathfrak{u},\alpha}_{t,x,y}(T)\ge
\widetilde g
\bigl(X^{\alpha
}_{t,x}(T)\bigr) \mbox{-a.s. }\ \forall\alpha\in
\Ac^{t}\bigr\}.
\]
Therefore, $\tilde{v}(t,x)=e^{ct}v(t,x)$. Since $\mu_{Y}^{\hat u}$ has
linear growth in its second argument~$y$, one can choose large enough
$c>0$ so that
%
\begin{equation}
\label{eq: muY_uhat_tilde} \tilde\mu_{Y}^{\hat u}\dvtx(t,x,y,z,a)
\mapsto
c y + e^{c t}\mu_{Y}^{\hat u}\bigl(t,x,e^{-c t}y,e^{-c t}z,a
\bigr)
\end{equation}
is nondecreasing in its $y$-variable. This means that these dynamics
satisfy the monotonicity assumption used in step~A above. Moreover, all
the assumptions needed to apply step A to this new problem are also satisfied.
Let
%
\begin{eqnarray}
\label{tilde_H} %
&& \widetilde{H}(t,x,y,p,M)\nonumber\\
&&\qquad:=
\sup_{a\in A} \biggl\{-cy-e^{ct}\mu^{\tilde u}_{Y}
\bigl(t,x,e^{-c t}y,e^{-c
t}\sigma_{X}(t,x,a)p,a\bigr)
\\
&&\hspace*{70pt}{}+\mu_{X}(t,x,a)^{\top} p + \frac{1}2\operatorname{Tr}
\bigl[\sigma_{X}\sigma_{X}^{\top}(t,x,a)M \bigr]
\biggr\},\nonumber
\end{eqnarray}
where $\tilde{u}$ is defined like $\hat{u}$ but now in terms of
$\tilde\sigma_{Y}$. We will denote by
$\widetilde{ \mathcal{U}}^+$ be the set of stochastic super-solutions of
%
\begin{eqnarray}
\label{eq: scale_sde} %
\varphi_t+
\widetilde{H}\bigl(\cdot,\vp,D\vp,D^{2}\vp\bigr)&=&0 \qquad\mbox{on }
\mathcal{D}_{<T},
\nonumber
\\[-8pt]
\\[-8pt]
\nonumber
\vp&=&\widetilde g\qquad\mbox{on } \mathcal{D}_{T}
\end{eqnarray}
and $\tilde{v}^+(t,x):=\inf_{w\in\widetilde{ \mathcal{U}}^+ } w(t,x)$.

From step A, we know that $\tilde{v}^+$ is a viscosity sub-solution of
the above PDE. Since any function $w(t,x)$ is a stochastic
super-solution of (\ref{HJB equation}) if and only if $\tilde
{w}(t,x)=e^{ct}w(t,x)$ is a stochastic super-solution of \eqref{eq:
scale_sde}, it follows that $\tilde{v}^+(t,x)=e^{ct}v^+(t,x)$. Now it
is easy to conclude that $v^{+}$ is a viscosity sub-solution of (\ref
{HJB equation}).

\textit{Step \textup{2} \textup{(}$v^{-}$ is the viscosity super-solution\textup{)}.}
Due to Assumption~\ref{ass:colunepty}, $v^-$ is well defined. Next we
will show that it satisfies the interior viscosity super-solution
property followed by the boundary condition.

\textit{Step \textup{2.1}. The interior super-solution property}: Let $(t_0,x_0)$
in the parabolic interior $[0,T)\times\mathbb{R}^d$ such that a smooth
function $\varphi$ strictly touches $v^-$ from below at $(t_0,x_0)$.
Assume by contradiction that
\[
\varphi_t +H\bigl(\cdot,\varphi, D\varphi, D^2\varphi
\bigr)>0 \qquad\mbox{at } (t_0, x_0).
\]
Hence there exists $a_0 \in A$, such that
%
\begin{equation}
\label{eq:a_0} \varphi_t +H^{u_0,a_0}\bigl(\cdot,\varphi, D
\varphi, D^2\varphi\bigr)>0\qquad \mbox{at } (t_0,
x_0),
\end{equation}
where $u_0=\hat{u}(t_0,x_0,\varphi(t_0,x_0),\sigma_X(t_0, x_0,
a_0)D\varphi(t_0,x_0), D^2\varphi(t_0,x_0))$
and
%
\begin{eqnarray}
\label{eq: H^{u,a}} &&H^{u,a}(t,x,y,p,M)
\nonumber
\\[-8pt]
\\[-8pt]
\nonumber
&&\qquad:=-\mu_{Y}(t,x,y,u,a)+
\mu_{X}(t,x,a)^{\top} p + \tfrac{1}2\operatorname{Tr}
\bigl[\sigma_{X}\sigma_{X}^{\top}(t,x,a)M \bigr].
\end{eqnarray}
From the continuity assumption on the coefficients in Assumption~\ref
{Assump: drift_and_volatility} and the continuity of $\hat u$ in
Assumption~\ref{ass: def hat u + regu}, there exists $\eps, \delta>0$
such that
%
%
%
\begin{eqnarray*}
&&\varphi_t +H^{u,a_0}\bigl(\cdot,y, D\varphi, D^2
\varphi\bigr)>0 
\qquad \forall(t,x)\in\overline{B(t_0,
x_0, \eps)}\\
&&\qquad \mbox{and } (y,u) \in R\times U \mbox{ s.t. }\bigl |y-
\vp(t,x)\bigr|\le\delta\\
&&\qquad\mbox{and }\bigl \|\sigma_Y(t,x,y,u,a_0)-
\sigma_X(t,x,a_0)D\varphi(t,x)\bigr\| \leq\delta.
%
\end{eqnarray*}

Now, on the compact torus $\mathbb{T}= \overline{B(t_0, x_0, \eps)}-
B(t_0, x_0, \eps/2)$, we have that $\varphi<v^-$ and the max of
$\varphi-v^-$ is attained since $v^{-}$ is LSC. Therefore, $\varphi
+\eta< v^- $ on $\mathbb{T}$ for some $\eta>0$. Since $w_n\nearrow
v^-$, a Dini-type argument shows that for large enough~$n$, we have
$\varphi+ \eta/2< w_n$ on $\mathbb{T}$ and $\varphi<w_n+\delta$ on
$\overline{B(t_0, x_0, \eps/2)}$. For simplicity, fix such an $n$ and
denote $w=w_n$. Now define for small $\kappa<\frac{\eta}{2} \wedge
\delta$,
\[
w^{\kappa}\triangleq\cases{ %
(\varphi+\kappa)\vee
w,&\quad $\mbox{on } \overline{B(t_0, x_0, \eps)},$
\vspace*{2pt}\cr
w,&\quad $\mbox{outside } \overline{B(t_0, x_0, \eps)}.$
}
\]
Since $w^{\kappa}(t_0,x_0)>v^-(t_0,x_0)$, we obtain a contradiction if
we can show that $w^{\kappa}\in\mathcal{U}^-$.

In order to do so, fix $t$ and $\{\tau^{\alpha}\} \in\mathfrak{S}^t $.
For a given $\mathfrak{u}\in\mfU(t)$ and $\alpha\in\Ac^t$, we will
construct an ``optimal'' $\widetilde{\alpha} \in\Ac^t$ in the
definition of stochastic sub-solutions for $w^{\kappa}$. We will divide
the construction into two cases:

\begin{longlist}[(ii)]
\item[(i)] $ w(\tau^{\alpha},X(\tau^{\alpha}))= w^{\kappa}(\tau^{\alpha
},X(\tau
^{\alpha}))$: Since $w$ is a stochastic sub-solution, there exists an
$\widetilde{\alpha}_1$ for $w$ in the definition which is ``optimal''
for the nature given $\mathfrak{u}$, $\alpha$ and $\tau^{\alpha}$. Let
$\widetilde{\alpha}=\widetilde{\alpha}_1$.

\item[(ii)] $w(\tau^{\alpha},X(\tau^{\alpha})) < w^{\kappa}(\tau^{\alpha
},X(\tau^{\alpha}))$: Let
\[
\theta_{1}^{\alpha}:=\inf\bigl\{s \in\bigl[
\tau^{\alpha},T\bigr]\dvtx\bigl(s, X_{t,x}^{\alpha\otimes_{\tau
^{\alpha}} a_0 }(s)
\bigr) \notin B(t_0, x_0, \eps/2) \bigr\}
\]
and
\[
\theta_2^{\alpha}:=\inf\bigl\{s \in\bigl[
\tau^{\alpha},T\bigr]\dvtx\bigl\llvert Y_{t,x,y}^{\mathfrak{u},
\alpha\otimes_{\tau^{\alpha}} a_0
}(s)-
\varphi\bigl(s,X_{t,x}^{\alpha\otimes_{\tau^{\alpha}} a_0
}(s)\bigr)\bigr\rrvert\geq\delta
\bigr\},
\]
with the convention that $\inf\varnothing=T$. Denote
$\theta^{\alpha}=\theta_1^{\alpha} \wedge\theta_2^{\alpha}$.
Then let
$\widetilde{\alpha}=a_0$ until~$\theta^{\alpha}$. Starting from
$\theta
^{\alpha}$, choose $\widetilde{\alpha}=\alpha^{*}$, where the
latter is
``optimal'' for nature given $\alpha$ and $\mathfrak{u}$ this time onward.
\end{longlist}

In summary, the above construction yields a candidate ``optimal''
control for $w^{\kappa}$ given by
\[
\widetilde{\alpha}= \bigl(\mathbh{1}_{A} \widetilde{
\alpha}_1 +\mathbh{1}_{A^c} \bigl(a_0
\mathbh{1}_{[t,\theta^{\alpha})} + \alpha^{*}\mathbh{1}_{[\theta
^{\alpha},T]}
\bigr) \bigr)\mathbh{1}_{[\tau^{\alpha}, T]},
\]
where
\[
A=\bigl\{w\bigl(\tau^{\alpha},X_{t,x}^{\alpha}\bigl(\tau^{\alpha}\bigr)\bigr)=w^{\kappa
}\bigl(\tau^{\alpha},X_{t,x}^{\alpha}
\bigl(\tau^{\alpha}\bigr)\bigr)\bigr\}.
\]
Let us check that what we constructed actually works: Let us abbreviate
\[
(X,Y)=\bigl(X_{t,x}^{\alpha\otimes_{\tau^{\alpha}} \widetilde
{\alpha} }, Y_{t,x,y}^{\mathfrak{u}, \alpha\otimes_{\tau^{\alpha
}}\widetilde
{\alpha} }
\bigr).
\]
Note that
%
\begin{eqnarray}
\label{eq: XY_processes on two disjoint sets} %
X(s)&=&\mathbh{1}_{A}X_{t,x}^{\alpha\otimes_{\tau^{\alpha}}
\widetilde
{\alpha}_1 }(s)
+ \mathbh{1}_{A^c}X_{t,x}^{\alpha\otimes_{\tau
^{\alpha
}} a_0 }(s)\qquad \mbox{for }
\tau^{\alpha}\leq s \leq\theta^{\alpha
},
\nonumber
\\[-8pt]
\\[-8pt]
\nonumber
Y(s)&=&\mathbh{1}_{A}Y_{t,x,y}^{\mathfrak{u}, \alpha\otimes_{\tau
^{\alpha}}
\widetilde{\alpha}_1 }(s) +
\mathbh{1}_{A^c}Y_{t,x,y}^{\mathfrak
{u}, \alpha
\otimes_{\tau^{\alpha}} a_0 }(s) \qquad\mbox{for }
\tau^{\alpha}\leq s \leq\theta^{\alpha}. %
\end{eqnarray}
Again for brevity, let us introduce the following sets:
\begin{eqnarray*}
E&=&\bigl\{Y\bigl(\tau^{\alpha}\bigr)< w^{\kappa}\bigl(\tau^{\alpha},X\bigl(\tau^{\alpha
}\bigr)\bigr)\bigr\},\qquad
E_0=E\cap A, \qquad E_1=E\cap A^c,
\\
G&=&\bigl\{Y(\rho)< w^{\kappa}\bigl(\rho,X(\rho)\bigr)\bigr\}, \qquad G_0=\bigl
\{Y(\rho)< w\bigl(\rho,X(\rho)\bigr)\bigr\}.
\end{eqnarray*}
Observe that
\[
E=E_0\cup E_1, \qquad E_0\cap
E_1=\varnothing\quad \mbox{and}\quad  G_0\subset G.
\]

The proof will be complete if we can show that $P(G|B)>0$ for any
nonnull set $ B\subset E$. In fact, it suffices to show that $\mathbb
{P}(G\cap B)>0$.
Relying on the decomposition $\mathbb{P}(G\cap B)=\mathbb{P}(G\cap
B\cap E_0)+\mathbb{P}(G\cap B \cap E_1)$ (recall that $B\subset E$), we
will divide the proof into two steps:

(i) $\mathbb{P}(B\cap E_0)>0$: Directly from the way
$\widetilde{\alpha}_1$ is defined, the definition of the stochastic
sub-solutions and $B\cap E_0 \subset A$, we get
\[
\mathbb{P}(G_0|B\cap E_0)=\mathbb{P}
\bigl(Y_{t,x,y}^{\mathfrak{u},\alpha
\otimes_{\tau
^{\alpha}} \widetilde{\alpha}_1}(\rho)<w\bigl(\rho
,X_{t,x}^{\alpha
\otimes
_{\tau^{\alpha}} \widetilde{\alpha}_1}(\rho)\bigr)|B\cap E_0\bigr)>0.
\]
This further implies that $\mathbb{P}(G\cap B\cap E_0) \geq\mathbb
{P}(G_0\cap B\cap E_0)>0$.

(ii) $\mathbb{P}(B\cap E_1)>0$: From \eqref{eq: XY_processes
on two disjoint sets} and $B\cap E_1\subset A^c$,
\begin{eqnarray*}
&&\mathbb{P}\bigl(Y\bigl(\theta^{\alpha}\bigr)<w^{\kappa}\bigl(\theta^{\alpha},X\bigl(\theta^{\alpha}\bigr)\bigr)|B\cap
E_1\bigr)\\
&&\qquad=\mathbb{P}\bigl(Y_{t,x,y}^{\mathfrak{u}, \alpha
\otimes_{\tau
^{\alpha}} a_0}\bigl(\theta^{\alpha}\bigr)<w^{\kappa}\bigl(\theta^{\alpha
},X_{t,x}^{\alpha\otimes_{\tau^{\alpha}} a_0 }
\bigl(\theta^{\alpha}\bigr) \bigr)|B\cap E_1\bigr).
\end{eqnarray*}
The analysis in \cite{Bouchard_Nutz_TargetGames} shows that
\[
\Delta(s)=Y\bigl(s\wedge\theta^{\alpha}\bigr)- \bigl(\varphi
\bigl(s\wedge
\theta^{\alpha},X\bigl(s\wedge\theta^{\alpha}\bigr)\bigr
)+\kappa
\bigr)
\]
is a super-martingale up to a change of measure. We will summarize
these arguments here: Let
\begin{eqnarray*}
\lambda(s)&:=&\sigma_{Y}\bigl(s,X(s),Y(s),\mathfrak
{u}[a_{0}]_{s},a_{0}\bigr)-\sigma
_{X}\bigl(s,X(s),a_{0}\bigr)D\vp\bigl(s,X(s)\bigr),
\\
\beta(s)&:= &\bigl(\varphi_t\bigl(s,X(s)\bigr)+H^{\mathfrak
{u}[a_{0}]_{s},a_{0}}
\bigl(s,X(s), Y(s),D\vp\bigl(s,X(s)\bigr),D^{2}\vp\bigl
(s,X(s)\bigr)
\bigr) \bigr)\\
&&{}\times\bigl \|\lambda(s)\bigr\| ^{-2}\lambda(s) \mathbh{1}_{\{\|
\lambda(s)\|> \delta\}}.
\end{eqnarray*}
From the definition of $\theta^{\alpha}$ and the regularity
and growth conditions in Assumptions~\ref{Assump: drift_and_volatility}
and \ref{assump: relative_growth_condition_mu_to_sigma}, $\beta$ is
uniformly bounded on
$[\tau^{\alpha}, \theta^{\alpha}]$. This ensures that the positive
exponential local martingale $M$ defined by the SDE
\[
M(\cdot)=1+\int_{\tau^{\alpha}}^{\cdot\wedge\theta^{\alpha}}
M(s)\beta
_{s}^{\top} \,dW_s
\]
is a true martingale. An application of It\^{o}'s formula immediately
implies that $M\Delta$ is a local super-martingale. By the definition
of $\theta^{\alpha}$, $\Delta$ is bounded by $-\delta-\kappa$ from
below and by $\delta-\kappa$ from above on $[\tau^{\alpha},\theta
^{\alpha}]$. Therefore, $M\Delta$ is bounded above by a martingale
$2M\delta$, and below by another martingale $-2M\delta$. An
application of Fatou's lemma implies that $M\Delta$ is a super-martingale.

From the definition of $E_1$ and $w^{\kappa}$, $\Delta(\tau^{\alpha
})<0$ on $B\cap E_1$. The super-martingale property of $M\Delta$
implies that there exists a nonnull $H \subset B \cap E_1$, $H \in\Fc
^t_{\tau^{\alpha}}$ such that $\Delta(\theta^{\alpha}\wedge\rho
)<0$ on
$H$. Therefore, from the decomposition
\begin{eqnarray*}
\Delta\bigl(\theta^{\alpha}\wedge\rho\bigr)\mathbh{1}_{H}&=&
\bigl(Y\bigl(\theta_1^{\alpha}\bigr)-\bigl(\varphi\bigl(\theta_1^{\alpha},X\bigl(\theta_1^{\alpha
}
\bigr)\bigr)+\kappa\bigr) \bigr)\mathbh{1}_{H\cap\{\theta
_1^{\alpha}<\theta_2^{\alpha
}\wedge
\rho\}}\\
&&{}+\bigl(Y\bigl(\theta_2^{\alpha}\bigr)-\bigl(\varphi\bigl(\theta_2^{\alpha},X\bigl(\theta_2^{\alpha
}
\bigr)\bigr)+\kappa\bigr) \bigr)\mathbh{1}_{H\cap\{\theta
_2^{\alpha}\leq\theta
_1^{\alpha}\wedge\rho\}}\\
&&{}+ \bigl(Y(\rho)
-\bigl(\varphi\bigl(\rho,X(\rho)\bigr)+\kappa\bigr) \bigr)\mathbh
{1}_{H\cap\{\rho<\theta^{\alpha} \}},
\end{eqnarray*}
we see that
%
\begin{eqnarray}
\label{eqn_firstpart_of_delta} Y\bigl(\theta_1^{\alpha}\bigr
)-\bigl(\varphi
\bigl(\theta_1^{\alpha},X\bigl(\theta_1^{\alpha
}
\bigr)\bigr)+\kappa\bigr)&<&0 \qquad\mbox{on } H\cap\bigl\{\theta
_1^{\alpha}<
\theta_2^{\alpha}\wedge\rho\bigr\},
\\
\label{eqn_secondpart_of_delta} Y\bigl(\theta_2^{\alpha}\bigr
)-\bigl(\varphi
\bigl(\theta_2^{\alpha},X\bigl(\theta_2^{\alpha
}
\bigr)\bigr)+\kappa\bigr)&<&0\qquad \mbox{on } H\cap\bigl\{\theta
_2^{\alpha}
\leq\theta_1^{\alpha}\wedge\rho\bigr\}
\end{eqnarray}
and that
%
\begin{equation}
\label{eqn_thirdpart_of_delta} Y(\rho)-\bigl(\varphi\bigl(\rho
,X(\rho)\bigr)+\kappa\bigr)<0\qquad
\mbox{on } H\cap\bigl\{ \rho<\theta^{\alpha} \bigr\}.
\end{equation}
On the one hand, on $H\cap\{\theta_1^{\alpha}<\theta_2^{\alpha
}\wedge
\rho\}$, $\varphi(\theta_1^{\alpha},X(\theta_1^{\alpha}))+\kappa
< w
(\theta_1^{\alpha},X(\theta_1^{\alpha}))$. Then from \eqref
{eqn_firstpart_of_delta}, we will have
%
\begin{equation}
\label{Y_below_w_part1} Y\bigl(\theta_1^{\alpha}\bigr)<w\bigl(\theta_1^{\alpha},X\bigl(\theta_1^{\alpha}
\bigr)\bigr) \qquad\mbox{on } H\cap\bigl\{\theta_1^{\alpha}<
\theta_2^{\alpha}\wedge\rho\bigr\}.
\end{equation}
On the other hand, on $H\cap\{\theta_2^{\alpha}\leq\theta
_1^{\alpha
}\wedge\rho\}$, we get $ Y(\theta_2^{\alpha})-\varphi(\theta
_2^{\alpha
},X(\theta_2^{\alpha}))=-\delta$. [The right-hand side cannot be equal
to $\delta$; otherwise $(\ref{eqn_secondpart_of_delta})$ would be
contradicted.] Recalling the fact that $\varphi<w+\delta$ on
$\overline
{B(t_0, x_0, \eps/2)}$, this observation gives that
%
\begin{eqnarray}
\label{Y_below_w_part2} Y\bigl(\theta_2^{\alpha}\bigr)-w\bigl(\theta_2^{\alpha},X\bigl(\theta_2^{\alpha
}
\bigr)\bigr)=(\varphi-w)
 \bigl(\theta_2^{\alpha},
 X\bigl(\theta_2^{\alpha}\bigr)\bigr)-\delta<0
 \nonumber
 \\[-8pt]
 \\[-8pt]
\eqntext{ \mbox{on } H\cap\bigl\{
\theta_2^{\alpha}\leq\theta_1^{\alpha}\wedge
\rho\bigr\}.}
\end{eqnarray}
We have obtained in $\eqref{Y_below_w_part1}$ and $\eqref
{Y_below_w_part2}$ that
\[
Y\bigl(\theta^{\alpha}\bigr)<w\bigl(\theta^{\alpha},X\bigl(\theta^{\alpha}\bigr)\bigr) \qquad\mbox{on } H\cap\bigl\{\theta
^{\alpha}\leq
\rho\bigr\}.
\]
Now from the definition of stochastic sub-solutions and of $\alpha^*$,
we have that
%
\begin{equation}
\label{conlcusion_on_theta_less_than_rho} \mathbb{P}\bigl(G_0|H\cap
\bigl\{
\theta^{\alpha}\leq\rho\bigr\}\bigr)>0\qquad \mbox{if } \mathbb
{P}\bigl(H\cap\bigl
\{\theta^{\alpha}\leq\rho\bigr\}\bigr)>0.
\end{equation}
On the other hand, $\eqref{eqn_thirdpart_of_delta}$ implies that
%
\begin{equation}
\label{conlcusion_on_theta_larger_than_rho} \mathbb{P}\bigl(G|H\cap
\bigl\{\theta^{\alpha}> \rho\bigr\}
\bigr)>0 \qquad\mbox{if } \mathbb{P}\bigl(H\cap\bigl\{\theta^{\alpha
}> \rho\bigr\}
\bigr)>0.
\end{equation}
Since $\mathbb{P}(H)>0, G_0\subset G$, and $H\subset E_1\cap B$,
\eqref
{conlcusion_on_theta_less_than_rho} and \eqref
{conlcusion_on_theta_larger_than_rho} imply $\mathbb{P}(G\cap E_1
\cap B)>0$.

\textit{Step \textup{2.2.} The boundary condition}:

Assume that for some $x_{0} \in\mathbb{R}^d$, we have
%
\begin{equation}
\label{eq: step2.2_contra at the boundary} v^{-}(T,x_{0})<g(x_{0}).
\end{equation}
Since $g$ is LSC, then from (\ref{eq: step2.2_contra at the boundary})
there exists $\eps>0$ such that
%
\begin{equation}
\label{eq: step2.2_contra at the boundary conclusion}
v^{-}(T,x_{0})<g(x)-\eps\qquad\mbox{for }
|x-x_0|\leq\eps.
\end{equation}
Since $v^{-}$ is LSC, then $v^{-}$ is bounded below on the compact
(rectangular) torus $\mathbb{T}=\overline{B(T,x_{0};\eps
)}-B(T,x_{0};\eps/2)$. Choose $\beta>0$ small enough, such that
\[
v^{-}(T,x_{0})-\frac{\eps^2}{4\beta}<\inf
_{\mathbb
{T}}v^{-}(t,x)-\eps.
\]
By a Dini-type argument, there exists a $w\in\Uc^{-}$, such that
%
\begin{equation}
\label{Eq: 2.2_Dini_Argument} v^{-}(T,x_{0})-\frac{\eps^2}{4\beta
}<\inf
_{\mathbb{T}}w(t,x)-\eps.
\end{equation}
We now define for $C>0$,
\[
\varphi^{\beta,C}=v^{-}(T,x_0)-\frac{|x-x_0|^2}{\beta}-C
(T-t).
\]
For any $a_0$ we can choose large enough $C$,\setcounter{footnote}{1}\footnote{Similar analysis
for \eqref{eq: contra 2_proof1.2_stepA} will guarantee that choosing
$C$ is possible. }
\[
\varphi_t^{\beta,C} +H^{u_0,a_0}\bigl(\cdot,
\varphi^{\beta,C}, D\varphi^{\beta,C}, D^2
\varphi^{\beta,C}\bigr)>0 \qquad\mbox{on }  \overline{B(T,x_{0};
\eps)},
\]
where $H^{u,a}$ is the same as that in \eqref{eq: H^{u,a}}, $u_0=\hat
{u}(T,x_0,\varphi(T,x_0),\sigma_X(T, x_0, \break a_0) D\varphi(T, x_0), a_0)$.
Then from the continuity of the coefficients in Assumption~\ref
{Assump: drift_and_volatility} and the continuity of $\hat u$ in
Assumption~\ref{ass: def hat u + regu}, for any $a_0$, and there exists
a small enough $\delta>0$ such that
\begin{eqnarray*}
\label{eq: super sol strict local proof super sol v*} 
&&\varphi_t^{\beta,C}
+H^{u,a_0}\bigl(\cdot,y, D\varphi^{\beta,C}, D^2\varphi
^{\beta,C}\bigr)>0 
\qquad \forall(t,x)\in\overline{B(T,
x_0, \eps)}\\
&&\qquad \mbox{and } (y,u) \in R\times U \mbox{ s.t. } \bigl|y-
\vp^{\beta,C}(t,x)\bigr|\le\delta\\
&&\qquad\mbox{and } \bigl\|\sigma_Y(t,x,y,u,a_0)-
\sigma_X(t,x,a_0)D\varphi^{\beta,C}(t,x)\bigr\|\leq
\delta. 
\end{eqnarray*}
Choosing $C$ at least as large as $\eps/2 \beta$, we obtain from
\eqref
{Eq: 2.2_Dini_Argument} that
\[
\vp^{\beta,C}\leq w-\eps\qquad\mbox{on } \mathbb{T}.
\]
Also we have that
%
\begin{equation}
\label{eq: comparison_varphi_and_g_2.2} \vp^{\beta, C}(T,x)\leq
v^{-}(T,x_0)<g(x)-
\eps\qquad\mbox{for } |x-x_{0}|\leq\eps.
\end{equation}
Now for $\kappa<\eps\wedge\delta$ define
%
\begin{equation}
\label{eq: w_kappa_2.2} w^{\beta,C,\kappa}\triangleq\cases{ %
 \bigl(\vp^{\beta,C} +\kappa\bigr)\vee w,&\quad $\mbox{on } \overline{ B(T,
x_{0}, \eps)},$
\vspace*{2pt}\cr
w,&\quad $\mbox{outside } \overline{ B(T, x_0, \eps)}.$}
\end{equation}
From \eqref{eq: comparison_varphi_and_g_2.2} and \eqref{eq:
w_kappa_2.2} it is easy to see that $w^{\beta,C,\kappa}(T,x)\leq g(x)$.
By applying arguments similar to those in step 2.1, we can show that
$w^{\beta,C,\kappa}$ is a stochastic sub-solution with $ w^{\beta
,C,\kappa}(T,x_{0})>v^{-}(T,x_{0}) $. This contradicts the definition
of $v^{-}$.
\end{pf}
To characterize $v$ as the unique viscosity solution of \eqref{HJB
equation}, we need a comparison principle.
%
\begin{proposition}[(Comparison principle)]\label{prop: comparison principle}
Under Assumptions \ref{Assump: drift_and_volatility}, \ref{ass: def hat
u + regu} and~\ref{ass: regu muY_hatu}, the comparison principle for
\eqref{HJB equation} holds. More precisely, let $U$ (resp.,~$V$) be a
bounded USC viscosity sub-solution (resp., LSC viscosity
super-solution) to \eqref{HJB equation}. If $U\leq V$ on $\mathcal
{D}_{T}$, then $U
\leq V$ on $\D$.
\end{proposition}
\begin{pf}
\textit{Step} 1: Without loss of generality, assume that
%
\begin{eqnarray}
\label{eq: H_decreasing_in_y}&& \exists\gamma>0, \mbox{ such that }
H(t,x,y,p,M)-H
\bigl(t,x,y',p,M\bigr)< -\gamma\bigl(y-y'\bigr)
\nonumber
\\[-8pt]
\\[-8pt]
\nonumber
&&\qquad\mbox{for all } y>y'.
\end{eqnarray}
Otherwise, let $\widetilde{U}(t,x)=e^{ct}U(t,x)$ and $\widetilde
{V}(t,x)=e^{ct}V(t,x)$. Then a straightforward
calculation shows that $\widetilde{U}$ (resp., $\widetilde{V}$) is a
sub-solution (resp., super-solution) to
%
\begin{eqnarray}
-\varphi_t-\widetilde{H}\bigl(\cdot,
\vp,D\vp,D^{2}\vp\bigr)&=&0 \qquad\mbox{on } \mathcal{D}_{<T},
\nonumber
\\[-8pt]
\\[-8pt]
\nonumber
\vp&=&\widetilde g\qquad\mbox{on } \mathcal{D}_{T},
\end{eqnarray}
where $\widetilde{g}(x)=e^{cT}g(x)$ and $ \widetilde{H}$ is the same
as that in \eqref{tilde_H}. We can choose $c$ large enough such that
\eqref{eq: H_decreasing_in_y} holds for $\widetilde{H}$. In fact, from
the Lipschitz continuity of $\mu_Y^{\hat u}$ in Assumption~\ref{ass:
regu muY_hatu}, for $y>y'$,
\begin{eqnarray*} &&\widetilde{H}^{a}(t,x,y,p,M)-
\widetilde{H}^{a}\bigl(t,x,y',p,M\bigr)
\\
&&\qquad=-c\bigl(y-y'\bigr)+ e^{ct} \bigl(\mu^{\tilde u}_{Y}
\bigl(t,x,e^{-c t}y',e^{-c
t}\sigma
_{X}(t,x,a)p,a\bigr)\\
&&\hspace*{111pt}{}-\mu^{\tilde u}_{Y}
\bigl(t,x,e^{-c t}y,e^{-c t}\sigma_{X}(t,x,a)p,a
\bigr) \bigr)
\\
&&\qquad \leq-c\bigl(y-y'\bigr)+e^{ct}L\cdot e^{-ct}
\bigl(y-y'\bigr)
\\
&&\qquad=-(c-L) \bigl(y-y'\bigr),
\end{eqnarray*}
where $L$ is the Lipschitz constant and
\begin{eqnarray*}
\label{tilde_H_a} \widetilde{H}^{a}(t,x,y,p,M)&:=& -cy-e^{ct}
\mu^{\tilde u}_{Y}\bigl(t,x,e^{-c t}y,e^{-c t}
\sigma_{X}(t,x,a)p,a\bigr)\\
&&{} +\mu_{X}(t,x,a)^{\top}
p + \tfrac{1}2\operatorname{Tr} \bigl[\sigma_{X}
\sigma_{X}^{\top}(t,x,a)M \bigr].
\end{eqnarray*}
Then $\gamma:=c-L>0$ for large enough c. Since $\widetilde{H}(\cdot
)=\sup_{a\in A}\widetilde{H}^{a}(\cdot)$,
equation~\eqref{eq: H_decreasing_in_y} holds for $\widetilde{H}$.

\textit{Step} 2:
In this step, we claim that for large enough $\lambda$, $V_{\delta
}:=V+\delta e^{-\lambda t}(1+|x|^2)$ is a LSC viscosity super-solution
to \eqref{HJB equation} for $\delta>0$. Then, if we can show that
$U-V_{\delta}\leq0$ on $\D$ for all $\delta>0$, we will get the
required result by sending $\delta$ to zero. Now we prove the above claim.

Obviously, the boundary condition is satisfied. Let $\varphi$ be a
smooth function which strictly touches $V_{\delta}$ from below at $
(t_0,x_0)\in\mathcal{D}_{<T}$. Let $\varphi^{\delta}=\varphi
-\delta e^{-\lambda
t}(1+|x|^2)$. Then $V-\varphi^{\delta}$ has a strict minimum at
$(t_0,x_0)$. Since $V$ is a viscosity super-solution, then it holds that
%
\begin{equation}
\label{eq: supersolution property of varphi_eps} \varphi^{\delta
}_t+H\bigl(t,x,
\varphi^{\delta},D\varphi^{\delta}, D^2\varphi
^{\delta}\bigr)\leq0.
\end{equation}
Note that
%
\begin{eqnarray}
\label{eq: derivative_varphi_eps} \varphi^{\delta}_t&=&\varphi
_t+\lambda
\delta e^{-\lambda t}\bigl(1+|x|^2\bigr),\qquad D\varphi^{\delta}=D
\varphi-2\delta e^{-\lambda t}x,
\nonumber
\\[-8pt]
\\[-8pt]
\nonumber
 D^2\varphi^{\delta}&=&D^2
\varphi-2\delta e^{-\lambda t}I_{d\times d}.
\end{eqnarray}
Consider the difference of $H(t,x,\varphi^{\delta},D\varphi^{\delta
},D^2\varphi^{\delta})$ and $H(t,x,\varphi,D\varphi, D^2\varphi)$. From
\eqref{eq: derivative_varphi_eps} and Assumption~\ref{Assump:
drift_and_volatility}, we get
%
\begin{eqnarray}\qquad
\label{eq: H-H_eps1} \bigl\llvert\mu_X^{\top}(t,x,a)D
\varphi(t,x)-\mu_X^{\top
}(t,x,a)D\varphi
^{\delta}(t,x)\bigr\rrvert&\leq& K\bigl|D\varphi(t,x)-D\varphi^{\delta
}(t,x)\bigr|
\nonumber
\\[-8pt]
\\[-8pt]
\nonumber
&=&2K
\delta e^{-\lambda t}|x|.
\end{eqnarray}
Similarly,
%
\begin{eqnarray}
\label{eq: H-H_eps2} &&\bigl\llvert\tfrac{1}{2}\operatorname
{Tr}\bigl(\sigma_X\sigma^{\top
}_{X}(t,x,a)
\bigr)D^2\varphi(t,x)-\tfrac{1}{2}\operatorname{Tr}\bigl(\sigma_X\sigma^{\top
}_{X}(t,x,a)
\bigr)D^2\varphi^{\delta}(t,x)\bigr\rrvert
\nonumber
\\[-8pt]
\\[-8pt]
\nonumber
&&\qquad\leq
K^2\,d\delta e^{-\lambda t}.
\end{eqnarray}
From the Lipschitz continuity of $\mu_Y^{\hat u}$ in Assumption~\ref
{ass: regu muY_hatu},
%
\begin{eqnarray}
\label{eq: H-H_eps3} &&\bigl\llvert\mu_Y^{\hat u}\bigl(t,x,\varphi,
\sigma_X(t,x,a)D\varphi,a\bigr)-\mu_Y^{\hat u}
\bigl(t,x,\vp^{\delta}, \sigma_X(t,x,a)D\vp^{\delta
},a
\bigr)\bigr\rrvert
\nonumber
\\[-8pt]
\\[-8pt]
\nonumber
&&\qquad\leq L\bigl(\delta e^{-\lambda t}\bigl(1+|x|^2
\bigr)+ 2K\delta e^{-\lambda t}|x|\bigr).
\end{eqnarray}
From \eqref{eq: H-H_eps1}, \eqref{eq: H-H_eps2} and \eqref{eq: H-H_eps3},
\begin{eqnarray*}
&&\bigl\llvert H\bigl(t,x,\varphi^{\delta},D\varphi^{\delta},D^2
\varphi^{\delta
}\bigr)-H\bigl(t,x,\varphi,D\varphi,D^2\varphi
\bigr)\bigr\rrvert\\
&&\qquad\leq\delta e^{-\lambda
t}\bigl(1+|x|^2\bigr)
\bigl(L+LK+K^2d+K\bigr).
\end{eqnarray*}
Taking $\lambda>\lambda^*:=L+LK+K^2d+K$,
from the above inequality, we get
\begin{eqnarray*}
\vp_t+H\bigl(t,x,\vp,D\vp,D^2\vp\bigr)&\leq&
\vp_t^{\delta}+H\bigl(t,x,\varphi^{\delta
},D
\varphi^{\delta},D^2\varphi^{\delta}\bigr) -\lambda\delta
e^{-\lambda t}\bigl(1+|x|^2\bigr)
\\
&&{}+\bigl\llvert H\bigl(t,x,\varphi^{\delta},D\varphi^{\delta},D^2
\varphi^{\delta
}\bigr)-H\bigl(t,x,\varphi,D\varphi,D^2\varphi
\bigr)\bigr\rrvert
\\
&\leq& \vp_t^{\delta
}+H\bigl(t,x,\varphi^{\delta},D
\varphi^{\delta},D^2\varphi^{\delta
}\bigr)\leq0.
\end{eqnarray*}

\textit{Step} 3:
In this step, we show that $U-V_{\delta}\leq0$ on $\D$ for all
$\delta
>0$. From boundedness of $U$ and $V$, for all $\delta>0$,
%
\begin{equation}
\label{eq: -infty when x approaches infty } \lim_{|x|\rightarrow
\infty}\sup_{[0,T]}(U-V_{\delta})
(t,x)=-\infty.
\end{equation}
This implies the supremum of $U-V_{\delta}$ on $\D$ is attained on $[0,
T]\times\mathcal{O}$ for some open bounded set $\mathcal{O}$ of $\R^d$.
We assume
\[
M^{*}:= \sup_{\D}(U-V_{\delta}) = \max
_{[0,T)\times\mathcal
{O}}(U-V_{\delta})> 0,
\]
and we will obtain a contradiction to the above equation. We consider a
bounded sequence $(t_{\eps}, s_{\eps},x_{\eps},y_{\eps})_{\eps}$ that
maximizes $\Phi_{\eps}$ on $[0, T]^2 \times\R^d \times\R^d$ with
$\Phi
_{\eps}=U(t, x)- V_{\delta} (s, y)- \phi_{\eps}(t, s, x,y)$ and
$\phi
_{\eps}(t,s,x,y):=\frac{1}{2\eps}(|t-s|^2+|x-y|^2)$. By arguments
similar to those in Theorem~4.4.4 of \cite{Pham_book}, we know that
$(t_{\eps}, s_{\eps},x_{\eps},y_{\eps})_{\eps}$ converges to
$(t_0,t_0,x_0,x_0)$ for some $(t_0,x_0)\in[0,T]\times\mathcal{O}$ and
%
\begin{equation}
\label{eq: limits-of-doubling-function} M_{\eps}=\Phi(t_{\eps},
s_{\eps},x_{\eps},y_{\eps})
\rightarrow M^*\quad \mbox{and}\quad \phi_{\eps}(t_{\eps},
s_{\eps},x_{\eps},y_{\eps
})\rightarrow0.
\end{equation}
In view of Ishii's lemma (Lemma~\ref{lemma: Ishii's Lemma}), there
exist $M,N \in\mathcal{S}^d$ such that
\begin{eqnarray*}
\biggl(\frac{1}{\eps}(t_{\eps}-s_{\eps}),
\frac{1}{\eps}(x_{\eps
}-y_{\eps
}), M \biggr)&\in&
\overline{P}^{2,+}U(t, x),
\\
\biggl(\frac{1}{\eps}(t_{\eps}-s_{\eps}),
\frac{1}{\eps}(x_{\eps
}-y_{\eps
}), N \biggr)&\in&
\overline{P}^{2,-}V_{\delta}(t, x).
\end{eqnarray*}
From the viscosity sub-solution and super-solution characterization of
$U$ and $V_{\delta}$ in terms
of super-jets and sub-jets, we then have
\begin{eqnarray*}
-\frac{1}{\eps}(t_{\eps}-s_{\eps})-H\biggl(t_{\eps},x_{\eps
},U(t_{\eps
},x_{\eps}),
\frac{1}{\eps}(x_{\eps}-y_{\eps}), M\biggr)&\leq&0,
\\
-\frac{1}{\eps}(t_{\eps}-s_{\eps})-H\biggl(s_{\eps},y_{\eps
},V_{\delta
}(s_{\eps},y_{\eps}),
\frac{1}{\eps}(x_{\eps}-y_{\eps}), N\biggr)&\geq& 0.
\end{eqnarray*}
By subtracting the two inequalities above, we get
\[
H\biggl(t_{\eps},x_{\eps},U(t_{\eps},x_{\eps}),
\frac{1}{\eps}(x_{\eps
}-y_{\eps}), M\biggr)\geq H
\biggl(s_{\eps},y_{\eps},V_{\delta}(s_{\eps
},y_{\eps}),
\frac{1}{\eps}(x_{\eps}-y_{\eps}), N\biggr).
\]
Subtracting $H(t_{\eps},x_{\eps},V_{\delta}(s_{\eps},y_{\eps}),
\frac
{1}{\eps}(x_{\eps}-y_{\eps}), M)$ from both sides of the equation
above, we get
%
\begin{eqnarray}
\label{eq: comparison_proof} %
\mathrm{LHS}&:=& H
\biggl(t_{\eps},x_{\eps},U(t_{\eps},x_{\eps}),
\frac{1}{\eps
}(x_{\eps}-y_{\eps}), M\biggr)\nonumber\\
&&{}- H
\biggl(t_{\eps},x_{\eps},V_{\delta}(s_{\eps
},y_{\eps}),
\frac{1}{\eps}(x_{\eps}-y_{\eps}), M\biggr)\nonumber\\
&\geq&
H\biggl(s_{\eps},y_{\eps},V_{\delta}(s_{\eps},y_{\eps}),
\frac{1}{\eps
}(x_{\eps}-y_{\eps}), N\biggr)\\
&&{}-H
\biggl(t_{\eps},x_{\eps},V_{\delta}(s_{\eps
},y_{\eps
}),
\frac{1}{\eps}(x_{\eps}-y_{\eps}), M\biggr)\nonumber\\
&=:&
\mathrm{RHS}.\nonumber
\end{eqnarray}
On the one hand, since $U(t_\eps,x_\eps)-V_{\delta}(s_\eps,y_\eps
)\geq M^*$,
%
\begin{equation}
\label{eq: LHS} \mathrm{LHS}\leq-\gamma\bigl(U(t_\eps
,x_\eps)-V_{\delta}(s_\eps
,y_\eps)\bigr)\leq-\gamma M^{*}.
\end{equation}
On the other hand, applying inequality \eqref{eq: trace_inequality from
Ishii'e Lemma} to $C =\sigma_{X}(t_\eps,x_\eps,a)$ and $D =\sigma
_{X}(s_\eps,y_\eps,a)$, we get
\begin{eqnarray*} I_1&:=&\biggl\llvert\frac{1}{2}
\operatorname{Tr}\bigl[\sigma_X\sigma_X^{\top
}(t_\eps,x_\eps
,a)M\bigr]-\frac{1}{2}\operatorname{Tr}\bigl[\sigma_X
\sigma_X^{\top}(s_\eps,y_\eps,a)N
\bigr]\biggr\rrvert
\\
&\leq&\frac{3}{2\eps}\operatorname{Tr} \bigl[\bigl(\sigma_X(t_\eps
,x_\eps)-\sigma_X(s_\eps,y_\eps)
\bigr) \bigl(\sigma_X(t_\eps,x_\eps)-\sigma
_X(s_\eps,y_\eps)\bigr)^{\top}
\bigr]
\\
&\leq&\frac{1}{2\eps}O\bigl(|t_\eps-s_\eps|^2+|x_\eps-y_\eps
|^2\bigr)\rightarrow0.
\end{eqnarray*}
In the last inequality, we use \eqref{eq:
limits-of-doubling-function} and Lipschitz continuity of $\sigma_X$
(uniformly in~$a$). Therefore,
%
\begin{equation}
\label{eq: convergence of I1} I_1 \rightarrow0 \qquad\mbox{as } \eps
\rightarrow0,
\mbox{ uniformly in } a\in A.
\end{equation}
Similarly, from $\eqref{eq: limits-of-doubling-function}$ and Lipschitz
continuity of $\mu_X$ (uniformly in $a$)
%
\begin{eqnarray}
\label{eq: convergence of I2} I_2:=\biggl\llvert\frac{1}{\eps}
\mu_X^{\top}(t_{\eps},x_{\eps
},a)
(x_{\eps
}-y_{\eps})-\frac{1}{\eps}\mu_X^{\top}(s_{\eps},y_{\eps
},a)
(x_{\eps
}-y_{\eps})\biggr\rrvert\rightarrow0
\nonumber
\\[-8pt]
\\[-8pt]
 \eqntext{\mbox{uniformly
in } a\in A.}
\end{eqnarray}
From \eqref{eq: limits-of-doubling-function} and Lipschitz continuity
of $\sigma_X$ (Assumption~\ref{Assump: drift_and_volatility}) and
$\mu
_Y^{\hat u}$ (Assumption~\ref{ass: regu muY_hatu}), we get
\begin{eqnarray*} I_3&:=&\biggl\llvert
\mu^{\hat u}_{Y} \biggl(t_{\eps},x_{\eps},V_{\delta
}(s_{\eps
},y_{\eps}),
\sigma_{X}(t_{\eps},x_{\eps},a) \biggl(\frac{x_{\eps
}-y_{\eps
}}{\eps} \biggr),a \biggr)\\
&&{}-\mu^{\hat u}_{Y}
\biggl(s_{\eps},y_{\eps
},V_{\delta}(s_{\eps},y_{\eps}),
\sigma_{X}(s_{\eps},y_{\eps
},a) \biggl(\frac{x_{\eps}-y_{\eps}}{\eps} \biggr),a \biggr)\biggr\rrvert
\\
&\leq&\nu\bigl(|t_{\eps}-s_{\eps}|+|x_{\eps}-y_{\eps}|\bigr)+
\frac
{1}{2\eps}O\bigl(|t_\eps-s_\eps|^2+|x_\eps-y_\eps|^2
\bigr) \rightarrow0\qquad \mbox{as } \eps\rightarrow0,
\end{eqnarray*}
where $\nu(z)\rightarrow0$ as $z\rightarrow0$. The first term in the
last inequality above is the modulus of continuity of $\mu^{\hat
u}_{Y}$ in the variables $(t,x)$ (uniformly in $a$) and the second term
comes from similar arguments for $I_1$ and $I_2$. Therefore,
%
\begin{equation}
\label{eq: convergence of I3} I_3 \rightarrow0\qquad \mbox{uniformly in }
a\in A.
\end{equation}
Then \eqref{eq: convergence of I1}, \eqref{eq: convergence of I2} and
\eqref{eq: convergence of I3} imply that
%
\begin{equation}
\label{eq: RHS} \mathrm{RHS}\rightarrow0\qquad \mbox{as } \eps
\rightarrow0.
\end{equation}
From \eqref{eq: comparison_proof}, \eqref{eq: LHS} and \eqref{eq: RHS},
we obtain a contradiction.
\end{pf}

\begin{corollary}\label{coro: v unique viscosity solution}
If $g$ is continuous and Assumptions \ref{Assump: drift_and_volatility}--\ref{ass:colunepty}
hold, then $v$ is the
unique bounded continuous viscosity solution of \eqref{HJB equation}.
\end{corollary}
\begin{pf}
From Theorem \ref{thm: main theorem}, $v^{+}$ (resp., $v^{-}$) is a bounded USC
viscosity sub-solution (resp., LSC viscosity super-solution) to \eqref
{HJB equation}. Then $v^+(T,x)\leq g(x)\leq v^{-}(T,x) $. This implies
$v^+\leq v^{-}$ on $\D$ from Proposition~\ref{prop: comparison
principle}. Since $v^{+}\geq v\geq v^{-}$ by definition, $v^{+}= v=
v^{-}$. We have shown that $v$ is continuous and a bounded viscosity
solution of \eqref{HJB equation}.

To check the uniqueness, let $w$ be a bounded continuous viscosity
solution of \eqref{HJB equation}. Note that $w$ is a LSC viscosity
super-solution and $v$ is an USC viscosity sub-solution of \eqref{HJB
equation}. From Proposition~\ref{prop: comparison principle}, $v\leq
w$ on $\D$. Similarly, $w\leq v$ on $\D$. This implies $w=v$ on $\D$.
\end{pf}

From Theorem~\ref{thm: main theorem} and Corollary~\ref{coro: v unique
viscosity solution}, we obtain dynamic programming principle as a byproduct.
%
\begin{corollary}[(Dynamic programming principle)]\label{coro:DPP}
Assume $g$ is continuous and Assumptions~\ref{Assump:
drift_and_volatility}--\ref{ass:colunepty} hold. For any $(t,x)\in\D$,
the following two statements hold:
\begin{enumerate}
\item[DPP 1.] For any $y>v(t,x)$, there exists $\mathfrak
{u}\in\mfU
(t)$ such that for all $\alpha\in\Ac^t$ and $\theta\in\mathbb{S}^t$,
\[
Y_{t,x,y}^{\mathfrak{u},\alpha}(\theta)\geq v\bigl(\theta
,X_{t,x}^{\alpha}(\theta)\bigr).
\]
\item[DPP 2.] For any $y<v(t,x)$ and $\mathfrak{u}\in\mfU
(t)$, there
exists $\alpha\in\Ac^t$ such that for all $\theta\in\mathbb{S}^t$,
\[
\mathbb{P} \bigl(Y_{t,x,y}^{\mathfrak{u},\alpha}\geq v\bigl(\theta
,X_{t,x}^{\alpha
}(\theta)\bigr) \bigr)<1.
\]
\end{enumerate}
\end{corollary}
\begin{pf}
DPP 1: If $y>v(t,x)=v^{+}(t,x)$ (due to Corollary~\ref{coro: v unique
viscosity solution}), there exists a $w\in\mathcal{U}^+$ such that
$y>w(t,x)$. From the definition of stochastic super-solution, there
exists $\mathfrak{u}\in\mfU(t)$ such that
\[
Y_{t,x,y}^{\mathfrak{u},\alpha}(\theta)\geq w\bigl(\theta
,X_{t,x}^{\alpha}(\theta)\bigr)\geq v\bigl(\theta,X_{t,x}^{\alpha}(\theta)\bigr)
\]
for all $\theta\in\mathbb{S}^t$ and $\alpha\in\Ac^t$.

DPP 2: If $y<v(t,x)=v^{-}(t,x)=\sup_{w\in\mathcal{U}^-}w(t,x)$, there
exists a $w\in\mathcal{U}^-$ such that $y<w(t,x)$. From the definition
of stochastic sub-solution, for any $\mathfrak{u}\in\mfU(t)$, there
exits an
$\alpha\in\mathcal{A}^t$ such that
\[
\mathbb{P} \bigl(Y_{t,x,y}^{\mathfrak{u},\alpha}(\theta)< w\bigl
(\theta
,X_{t,x}^{\alpha}(\theta)\bigr) \bigr)>0
\]
for all $\theta\in\mathbb{S}^t$. Since $w(\theta,X_{t,x}^{\alpha
}(\theta
))\leq v(\theta,X_{t,x}^{\alpha}(\theta))$, this gives us the
desired result.
\end{pf}

\begin{appendix}
\section*{Appendix}\label{app}
\subsection{Proof of Proposition~\texorpdfstring{\protect\ref{prop:Upnemty}}{2.1}}
We carry out the proof in two steps. First under Assumptions~\ref{ass:
def hat u + regu} and \ref{ass: regu muY_hatu}, we will show that
there exists a classical solution to~\eqref{HJB equation}. Next, we
will show that if we additionally have Assumption~\ref{Assump:
drift_and_volatility}, then every classical super-solution is a
stochastic super-solution, which implies in particular that $\Uc^+$ is
not empty.

\textit{Step \textup{1.} Existence of a classical super-solution to \eqref{HJB
equation}:}

\emph{Step} 1A. In this step we will assume that $\mu_{Y}^{\hat
u}$ is nondecreasing in its $y$-variable.
Letting $\phi(t,x)=-e^{\lambda t}$ we have that
%
\setcounter{equation}{0}
\begin{equation}
\label{eq: constructing_classical_supersolution} \phi_t+H\bigl
(t,x,\phi,D\phi,D^2\phi
\bigr)=-\lambda e^{\lambda t}+\sup_{a\in
A}\bigl\{ -
\mu_Y^{\hat u}\bigl(t,x,\phi(t,x),0,a\bigr) \bigr\}.
\end{equation}
From the linear growth condition of $\mu_Y^{\hat u}$ in
Assumption~\ref
{ass: regu muY_hatu}, we know there exists an $L>0$, such that $-\mu
_Y^{\hat u}(t,x,\phi(t,x),0,a)\leq L(1+|\phi(t,x)|)=L(1+e^{\lambda
t})$. Therefore, from \eqref{eq: constructing_classical_supersolution},
\[
\phi_t+H\bigl(t,x,\phi,D\phi,D^2\phi\bigr)\leq-\lambda
e^{\lambda
t}+L\bigl(1+e^{\lambda t}\bigr)\leq0\qquad \mbox{in } \D, \mbox{ for }
\lambda> 2L.
\]
Fix $\lambda>2L$, and choose $N_2$ such that $-e^{\lambda T}+N_2\geq\|
g\|_{\infty}$. Then $\phi'(T,x)=\phi(T,x)+N_2 \geq g(x)$. From the
assumption that $\mu_{Y}^{\hat u}$ is nondecreasing in its
$y$-variable, it holds that
\[
\phi'_t+H\bigl(t,x,\phi',D
\phi',D^2\phi'\bigr)\leq0 \qquad\mbox{on }
\mathcal{D}_{<T}.
\]
Therefore, $\phi'$ is a classical super-solution.

\textit{Step} 1B. We now show the same result for more general
$\mu
_{Y}^{\hat u}$. This follows the same reparameterization argument
outlined in step 1.2B in the proof of the main theorem.

\textit{Step \textup{2.} Classical super-solutions are stochastic super-solutions.}
Let $w$ be a classical super-solution. Fix $(t,x,y)\in\D\times
\mathbb
{R}$ and $\{\tau^{\alpha}\} \in\mathfrak{S}^t$. Let $\overline{Y}$ be
the unique strong solution (which is thanks to Assumption~\ref{ass:
regu muY_hatu}) of the equation
\begin{eqnarray*}
\overline{Y}(l)&=&Y^{\mathfrak{u},\alpha}_{t,x,y}\bigl(\tau^\alpha
\bigr)\\
&&{}+\int_{\tau^\alpha
}^{\tau^\alpha\vee l}\mu_{Y}^{\hat{u}}
\bigl(s,X_{t,x}^{\alpha
}(s),\overline{Y}(s),
\sigma_{X}\bigl(s,X_{t,x}^{\alpha}(s),\alpha
_{s}\bigr)Dw\bigl(s,X_{t,x}^{\alpha}(s)\bigr),
\alpha_{s} \bigr) \,ds
\\
&&{}+ \int_{\tau^\alpha}^{\tau^\alpha\vee l} \sigma_{X}
\bigl(s,X_{t,x}^{\alpha
}(s),\alpha_{s}\bigr)Dw
\bigl(s,X_{t,x}^{\alpha}(s)\bigr)\,dW_{s},\qquad l \geq\tau
^{\alpha},
\end{eqnarray*}
for any $\mathfrak{u}\in\mfU(t)$ and $\alpha\in\Ac^t$, and set
$\overline
{Y}(s)=Y_{t,x,y}^{\mathfrak{u}, \alpha}(s) $ for $s <\tau^{\alpha
}$. We will
set $\tilde{\mathfrak{u}}$ to be
\[
\tilde{\mathfrak{u}}:=\tilde{\mathfrak{u}}[\alpha](s)=\hat{u}
\bigl(s,X_{t,x}^{\alpha
}(s),\overline{Y}(s), \sigma_{X}
\bigl(s,X_{t,x}^{\alpha}(s),\alpha_s\bigr)Dw
\bigl(s,X_{t,x}^{\alpha}(s)\bigr),\alpha_s\bigr).
\]
It is not difficult to check that $\tilde{\mathfrak{u}}\in\mfU(t,\{
\tau^{\alpha
}\})$. We will show that for any $\mathfrak{u}\in\mfU(t)$, $\alpha
\in\mathcal
{A}^t$ and each stopping time $\rho\in\mathbb{S}^t$, $\tau^{\alpha
}\leq\rho\leq T$ with the simplifying notation $X:= X_{t,x}^{\alpha},
Y:=Y_{t,x,y}^{\mathfrak{u}\otimes_{\tau^{\alpha}}\tilde{\mathfrak
{u}}[\alpha],\alpha}$,
we have
\[
Y(\rho)\geq w\bigl(\rho, X(\rho)\bigr)\qquad \mathbb{P}\mbox{-a.s.}
\mbox{ on } \bigl
\{Y\bigl(\tau^{\alpha}\bigr)> w\bigl(\tau^{\alpha}, X\bigl(\tau^{\alpha}\bigr)\bigr)\bigr\}.
\]
Note that $\overline{Y}= Y_{t,x,y}^{\mathfrak{u}\otimes_{\tau
^{\alpha}}\tilde
{\mathfrak{u}}[\alpha],\alpha}
$ for $s\geq\tau^{\alpha}$. We will carry out the rest of the proof in
two steps.

\textit{Step} 2A. In this step we will assume that $\mu_{Y}^{\hat
u}$ is nondecreasing in its $y$-variable. Let
\begin{eqnarray*}
A&=& \bigl\{Y\bigl(\tau^{\alpha}\bigr)> w\bigl(\tau^{\alpha},
X\bigl(\tau^{\alpha}\bigr)\bigr)\bigr\},\qquad  Z(s)=w\bigl(s,X(s)\bigr),\\
\Gamma(s)&=&
\bigl(Z(s)-Y(s) \bigr)\mathbh{1}_{A}.
\end{eqnarray*}
Therefore, for $s\geq\tau^{\alpha}$,
\begin{eqnarray*}
dY&=& \mu_{Y}^{\hat{u}} \bigl(s,X(s),Y(s), \sigma_{X}
\bigl(s,X(s),\alpha_{s}\bigr)Dw\bigl(s,X(s)\bigr),
\alpha_{s} \bigr)\,ds\\
&&{} + \sigma_{X}\bigl(s,X(s),\alpha
_{s}\bigr)Dw\bigl(s,X(s)\bigr)\,dW_{s},
\\
dZ&=& \bigl\lbrace w_t\bigl(s,X(s)\bigr)+\mu_{X}
\bigl(s,X(s),\alpha_s\bigr)^{\top} Dw\bigl(s,X(s)\bigr)\\
&&\hspace*{18pt}{} +
\tfrac{1}2\operatorname{Tr}\bigl[\sigma_{X}
\sigma_{X}^{\top
}\bigl(s,X(s),\alpha_s
\bigr)D^2w\bigl(s,X(s)\bigr)\bigr] \bigr\rbrace\, ds
\\
&&{}+\sigma_{X}\bigl(s,X(s),\alpha_{s}\bigr)Dw\bigl(s,X(s)
\bigr)\,dW_{s}.
\end{eqnarray*}
From above equations,
%
\begin{equation}
\label{eq: Gamma_integral} \Gamma(s)=\mathbh{1}_{A}\int_{\tau
^{\alpha}}^{s}
\bigl(\xi(u)-\gamma'(u)\bigr) \,du\qquad \mbox{for } s\geq
\tau^{\alpha},
\end{equation}
where
\begin{eqnarray*}
\gamma'&:=&
\mu_{Y}^{\hat{u}}\bigl(\cdot,X,w(\cdot,X),\sigma_{X}(\cdot,X,\alpha)Dw(\cdot,X),\alpha\bigr)-\mu_{X}(\cdot,X,
\alpha)^{\top} Dw(\cdot,X) \\
&&{}- \tfrac{1}2 \operatorname{Tr}\bigl[
\sigma_{X}\sigma_{X}^{\top}(\cdot,X,\alpha
)D^2w(\cdot,X)\bigr]-w_t(\cdot,X)
\end{eqnarray*}
and
\[
\xi:=\mu_{Y}^{\hat{u}} \bigl(\cdot,X,Z, \sigma_{X}(\cdot,X,\alpha)Dw(\cdot,X),\alpha\bigr) -\mu_{Y}^{\hat{u}}
\bigl(\cdot,X,Y, \sigma_{X}(\cdot,X, \alpha)Dw(\cdot,X),\alpha
\bigr).
\]
Since $w$ is a classical super-solution $\gamma'\geq0$. Then from
\eqref{eq: Gamma_integral} it follows that
\[
\Gamma(s)\leq\mathbh{1}_A\int_{\tau^{\alpha}}^{s}
\xi(u) \,du \quad\mbox{and} \quad\Gamma^{+}(s)\leq\mathbh{1}_A\int
_{\tau^{\alpha}}^{s} \xi^{+}(u) \,du\qquad \mbox{for } s
\geq\tau^{\alpha}.
\]
From the Lipschitz continuity of $\mu^{\hat{u}}_Y$ in $y$-variable in
Assumption~\ref{ass: regu muY_hatu},
\[
\label{eq: fun_gronwall} \Gamma^{+}(s)\leq\mathbh{1}_A \int
_{\tau^{\alpha}}^{s} \xi^{+}(u) \,du \leq\int
_{\tau^{\alpha}}^{s} L \Gamma^{+}(u) \,du \qquad\mbox{for }
s\geq\tau^{\alpha},
\]
where we also use the assumption that $\mu_Y^{\hat u}$ is
nondecreasing in its $y$-variable to obtain the second inequality.
Since $\mathbb{E}\Gamma^+(\tau^{\alpha})=0$, an application of
Gronwall's inequality implies that
$\mathbb{E}\Gamma^+(\rho)\leq0$.

\textit{Step} 2B: Now we will show the same result for more general
$\mu_{Y}^{\hat u}$. However, this again follows the same
reparameterization argument outlined in step 1.2B in the proof of the
main theorem. 

\subsection{Proof of Proposition~\texorpdfstring{\protect\ref{Prop: U+ notempty}}{2.4}}
Take $w(t,x)=m$ for any $(t,x)\in\D$, where the constant $m$ is a lower
bound of $g$. For any given $\mathfrak{u}\in\mfU(t),\alpha\in\Ac
^t$, choose any
$\widetilde{\alpha}\in\Ac^t$.
Let $B\subset\{Y(\tau^{\alpha})<w(\tau,X(\tau^{\alpha}))\}$ and
$\mathbb
{P}(B)>0$.
Set
\[
\theta_s \triangleq\cases{ %
\displaystyle\frac{\mu_{Y}\sigma_Y}{\|\sigma_{Y}\|^2}\bigl
(s,X(s),Y(s),\mathfrak{u}[\alpha\otimes_{\tau^{\alpha}}
\widetilde{\alpha}]_{s},[\alpha\otimes_{\tau^{\alpha
}}\widetilde{
\alpha}]_s\bigr),\vspace*{2pt}\cr
\hspace*{35pt}\mbox{if } \sigma_{Y}\bigl(s,X(s),Y(s),
\mathfrak{u}[\alpha\otimes_{\tau^{\alpha}}\widetilde{\alpha
}]_{s},[\alpha
\otimes_{\tau^{\alpha
}}\widetilde{\alpha}]_s\bigr)\neq0,
\vspace*{2pt}\cr
C,\qquad \mbox{otherwise},}
\]
for some constant vector C in $\R^d$. Therefore, $\theta_{s}$ satisfies
Novikov's condition due to Assumption~\ref{assum: stochastic
semisolution is nonemptyp}, and $\widetilde{W}(s)=W(s)-\int
_{0}^{s}\theta_{u}\,du$ is a Brownian
motion under the probability measure $\mathbb{Q}$, where
\begin{eqnarray*}
\mathbb{Q}(A)&=&\mathbb{E}_{\mathbb{P}}(Z_T \mathbh{1}_{A})
\qquad\mbox{for all } A\in\mathcal{F}\quad \mbox{and}\\
 Z_s&:=&\exp\biggl(\int
_0^s \theta_u
\,dW_u-\frac{1}{2}\int_0^s
\|\theta_u\|^2 \,du \biggr).
\end{eqnarray*}
$Z_T\in\L^q(\mathbb{P})$ for any $q\geq1$ since $\theta$ is a
bounded. From Assumption~\ref{assum: stochastic semisolution is
nonempty} and the assumption that $\sigma_Y$ is invertible in its
$u$-variable (Assumption~\ref{ass: def hat u + regu}), it follows that
$\sigma_Y(t,x,y,u,a)=0$ implies $\mu_Y(t,x,y,u,a)=0$. Therefore under
$\mathbb{Q}$
\[
dY(s)=\sigma_{Y}\bigl(s,X(s),Y(s),\mathfrak{u}[\widetilde{\alpha
}]_{s},\widetilde{\alpha}_s\bigr)\,d\widetilde{W}_{s}\qquad
\mbox{for } s\geq\tau^{\alpha},
\]
where $Y:=Y_{t,x,y}^{\mathfrak{u},\alpha\otimes_{\tau^{\alpha
}}\widetilde{\alpha
}}$. We will show that the $\mathbb{Q}$-local martingale $Y$ is
actually a $\mathbb{Q}$-martingale.
Assumption~\ref{Assump: drift_and_volatility} implies that
%
\begin{equation}
\label{eq: max_square_integrability} \mathbb{E}_{\mathbb{P}} \Bigl
[\sup_{0\leq s \leq T}\bigl|Y(s)\bigr|^2
\Bigr]<\infty;
\end{equation}
see, for example, Theorem~1.3.5 in \cite{Pham_book} or Theorem~2.2 in
\cite{Touzi_book}.
As a result an application of H\"older's inequality yields
%
\begin{eqnarray}
\label{eq: max_integrability under Q} \mathbb{E}_{\mathbb{Q}} \Bigl
[\sup_{0\leq s \leq T}\bigl|Y(s)\bigr|
\Bigr]&=& \mathbb{E}_{\mathbb{P}} \Bigl[\sup_{0\leq s \leq T}\bigl|Y(s)\bigr|
\cdot Z_T \Bigr]
\nonumber
\\[-8pt]
\\[-8pt]
\nonumber
&\leq&\mathbb{E}_{\mathbb{P}} \Bigl[\sup
_{0\leq s \leq T}\bigl|Y(s)\bigr|^2 \Bigr] \mathbb{E}_{\mathbb{P}}
\bigl[Z_T^2\bigr]<\infty.
\end{eqnarray}
From \eqref{eq: max_integrability under Q}, $Y$ is a martingale on
$[\tau^{\alpha}, T]$ under $\mathbb{Q}$. Moreover, since $\mathbb{Q}$
is equivalent to $\mathbb{P}$ we have $\mathbb{Q}(B)>0$. As a result of
the latter two statements, for any $\rho\geq\tau^{\alpha}$,
\[
Y(\rho)\leq Y\bigl(\tau^{\alpha}\bigr)\qquad \mbox{on some set }
H\subset B
\mbox{ with } \mathbb{Q}(H)>0.
\]
Since $H\subset B$,
\[
Y(\rho)\leq Y\bigl(\tau^{\alpha}\bigr)<m=w(t,x) \qquad\mbox{on } H.
\]
This implies $\mathbb{Q}(Y(\rho)<m|B)>0$ and by equivalence of the
measures $\mathbb{P}(Y(\rho)<m|B)>0$. Therefore, $w(t,x)=m$ is a
stochastic sub-solution. \hfill$\square$

\subsection{Some well-known results from the theory of viscosity solutions}
In this subsection, we introduce an alternative definition of viscosity
solutions and Ishii's lemma following \cite{Pham_book}. First, we
define the
second-order super-jet of an USC function $U$ at a
point $(\overline{t}, \overline{x})\in[0, T)\times\mathbb{R}^d$ as
the set of elements $(\overline{q}, \overline{p}, \overline{M})\in
\R
\times\R^d\times\mathcal{S}^d$
satisfying
\begin{eqnarray*}
U(t, x) &\leq& U(\overline{t}, \overline{x}) + \overline{q}(t -
\overline{t}) +
\overline{p}\cdot(x - \overline{x}) + \tfrac
{1}{2}M(x-\overline{x})\cdot(x-
\overline{x})\\
&&{}+o\bigl(|t- \overline{t}|+|x - \overline{x}|^2\bigr).
\end{eqnarray*}
This set is denoted by $P^{2,+}U(\overline{t}, \overline{x})$.
Similarly, $P^{2,-}V(\overline{t}, \overline{x})$, the
second-order sub-jet of a LSC function V at the
point $(\overline{t}, \overline{x})\in[0, T)\times\mathbb{R}^d$ is
defined as the set of elements $(\overline{q}, \overline{p},
\overline
{M})\in\R\times\R^d\times\mathcal{S}^d$
satisfying
\begin{eqnarray*}
V(t, x) &\geq& V(\overline{t}, \overline{x}) + \overline{q}(t -
\overline{t}) +
\overline{p}\cdot(x - \overline{x}) + \tfrac
{1}{2}M(x-\overline{x})\cdot(x-
\overline{x})\\
&&{}+o\bigl(|t- \overline{t}|+|x - \overline{x}|^2\bigr).
\end{eqnarray*}

For technical reasons related to
Ishii's lemma, we also need to consider the limiting super-jets and sub-jets.
More precisely, we define $\overline{P}^{2,+}U(t, x)$ as the set of
elements $(q,p,M)\in\R\times\R^d\times\mathcal{S}^d$ for
which there exists a sequence $(t_{\eps}, x_{\eps}, q_{\eps},
p_{\eps
},M_{\eps})_{\eps}$ satisfying $(q_{\eps}, p_{\eps},M_{\eps})\in
P^{2,+}U(t_\eps, x_\eps)$ and
$(t_{\eps}, x_{\eps},U(t_{\eps}, x_{\eps}),\break   q_{\eps}, p_{\eps
},M_{\eps
}) \rightarrow(t, x,U(t, x), q, p,M)$. The set $\overline{P}^{2,-}V(t,
x)$ is defined
similarly. Now we state the alternative definition of viscosity
solutions to \eqref{HJB equation}.
%
\begin{lemmaa}\label{lemma: alternative_def_of_viscosity_solution}
A USC (resp., LSC) function w on $\mathcal{D}_{<T}$ is a viscosity sub-solution
(resp., super-solution) to \eqref{HJB equation} if and only if for all
$(t, x)\in\mathcal{D}_{<T}$,
and all $(q, p,M) \in\overline{P}^{2,+}w(t, x)$ [resp., $\overline
{P}^{2,-}w(t, x)$],
\[
-q-H\bigl(t, x,w(t, x), p,M\bigr) \leq(\mbox{resp}., \geq)\ 0.
\]
\end{lemmaa}
Finally, we state Ishii's lemma used in \cite{Pham_book} without proof
and refer the reader to Theorem~8.3 in \cite{MR1118699}.
%
\begin{lemmaa}[(Ishii's lemma)]\label{lemma: Ishii's Lemma}
Let U (resp., V) be an USC (resp., LSC) function on $\mathcal{D}_{<T}$,
$\varphi\in
C^{1,1,2,2}([0, T)^2\times\R^d\times
\R^d)$, and $(t_0, s_0, x_0, y_0) \in[0, T)^2\times\R^d\times
\R^d$ a local maximum of $U(t, x)-V (s, y)-\varphi(t, s, x, y)$.
Then, for all $\eta> 0$, there exist $M, N \in\mathcal{S}^d$ satisfying
\begin{eqnarray*}
\bigl(\varphi_t(t_0, s_0,
x_0, y_0), D_x\varphi(t_0,
s_0, x_0, y_0), M \bigr)&\in&
\overline{P}^{2,+}U(t, x),
\\
\bigl(-\varphi_s(t_0, s_0,
x_0, y_0), -D_y\varphi(t_0,
s_0, x_0, y_0), N \bigr)&\in&
\overline{P}^{2,-}V(t, x)
\end{eqnarray*}
and
\[
\pmatrix{ M & 0
\vspace*{2pt}\cr
0 & N}
\leq D^2_{x,y}
\varphi(t_0, s_0, x_0,
y_0)+\eta\bigl(D^2_{x,y}
\varphi(t_0, s_0, x_0,
y_0) \bigr)^2.
\]
\end{lemmaa}

\begin{remark}
From Remark~4.4.9 in \cite{Pham_book}, by choosing $\varphi_{\eps
}(t,s,x,y):=\frac{1}{2\eps}(|t-s|^2+|x-y|^2)$ and $\eta=\eps$, for any
$d \times n$ matrices $C,D$, we get
%
\begin{equation}
\label{eq: trace_inequality from Ishii'e Lemma} \operatorname
{Tr}\bigl(CC^{\top}M- DD^{\top}N
\bigr)\leq\frac{3}{\eps}\operatorname{Tr}\bigl((C-D) (C-D)^{\top}\bigr).
\end{equation}
\end{remark}
\end{appendix}
%

%





\printaddresses
\end{document}